\documentclass[12pt,reqno]{amsart}

\usepackage{amsthm, amsmath, amsfonts, amssymb, graphicx, tikz-cd, MnSymbol, comment, hyperref, enumerate}

\usepackage[utf8]{inputenc}
\usepackage[normalem]{ulem}

\DeclareFontFamily{U}{mathb}{\hyphenchar\font45}
\DeclareFontShape{U}{mathb}{m}{n}{
<-6> mathb5 <6-7> mathb6 <7-8> mathb7
<8-9> mathb8 <9-10> mathb9
<10-12> mathb10 <12-> mathb12}{}
\DeclareSymbolFont{mathb}{U}{mathb}{m}{n}
\DeclareMathSymbol{\llcurly}{\mathrel}{mathb}{"CE}

\newtheorem{theorem}{Theorem}[section]
\newtheorem{lemma}[theorem]{Lemma}
\newtheorem{proposition}[theorem]{Proposition}
\newtheorem{corollary}[theorem]{Corollary}

\theoremstyle{definition}
\newtheorem{definition}[theorem]{Definition}

\newtheorem{remark}[theorem]{Remark}
\newtheorem{claim}[theorem]{Claim}

\newcommand{\func}[1]{\operatorname{#1}}

\newcommand{\dev}{{\sf DeV}}
\newcommand{\ba}{{\sf BA}}

\newcommand{\PBSp}{{\sf PBSp}}
\newcommand{\Sp}{{\sf Sp}}
\newcommand{\KHaus}{{\sf KHaus}}
\newcommand{\Stone}{{\sf Stone}}

\newcommand\RO{{\mathcal{RO}}}
\newcommand\Clop{{\sf Clop}}

\def\cl{{\sf cl}}
\def\int{{\sf int}}

\newcommand{\Id}{{\rm Id}}
\newcounter{num}

\makeatletter
\def\blfootnote{\xdef\@thefnmark{}\@footnotetext}
\makeatother

\begin{document}

\title{De Vries powers and proximity Specker algebras}

\author{G.~Bezhanishvili}
\address{New Mexico State University}
\email{guram@nmsu.edu}

\author{L.~Carai}
\address{University of Salerno\\University of Barcelona}
\email{luca.carai.uni@gmail.com}

\author{P.~J.~Morandi}
\address{New Mexico State University}
\email{pmorandi@nmsu.edu}

\author{B.~Olberding}
\address{New Mexico State University}
\email{bruce@nmsu.edu}

\date{}

\begin{abstract}
By de Vries duality \cite{deV62}, the category $\KHaus$ of compact Hausdorff spaces is dually equivalent to the category $\dev$ of de Vries algebras. In \cite{BMMO15b} an alternate duality for $\KHaus$ was developed, where de Vries algebras were replaced by proximity Baer-Specker algebras. The functor associating with each compact Hausdorff space a proximity Baer-Specker algebra was described by generalizing the notion of a boolean power of a totally ordered domain to that of a de Vries power. It follows that $\dev$ is equivalent to the category $\PBSp$ of proximity Baer-Specker algebras. The equivalence is obtained by passing through $\KHaus$, and hence is not choice-free. In this paper we give a direct algebraic proof of this equivalence, which is choice-free. To do so, we give an alternate choice-free description of de Vries powers of a totally ordered domain.

\bigskip

{\noindent}{\it 2020 Mathematics subject classification.} 06F25; 13G05; 54E05.

\medskip

{\noindent}{\it Keywords.} Proximity, de Vries algebra, Specker algebra, Baer ring, integral domain, boolean power.

\end{abstract}


\maketitle

\section{Introduction}

By the celebrated Stone duality, the category $\ba$ of boolean algebras is dually equivalent to the category $\Stone$ of Stone spaces (zero-dimensional compact Hausdorff spaces). The functor from $\Stone$ to $\ba$ associates with each Stone space $X$, the boolean algebra $\Clop(X)$ of clopen subsets of $X$. Thinking of clopen subsets of $X$ as continuous characteristic functions, we can identify $\Clop(X)$ with the idempotents of the $\mathbb R$-algebra $C(X)$ of all continuous real-valued functions on $X$. The $\mathbb{R}$-subalgebra of $C(X)$ generated by the idempotents of $C(X)$ is then the $\mathbb{R}$-algebra of finitely-valued continuous real-valued functions on $X$. The reals can be replaced by an arbitrary domain $D$ with the discrete topology, thus yielding the notion of a Specker $D$-algebra, which, as the $D$-algebra of (finitely-valued) continuous $D$-valued functions on a Stone space,
 is nothing more than a boolean power of $D$ (see \cite{BMMO15a}). Let $\Sp_D$ be the category of Specker $D$-algebras. Then $\Sp_D$ is dually equivalent to $\Stone$ \cite[Cor.~3.9]{BMMO15a}, hence $\Sp_D$ is equivalent to $\ba$, and we arrive at the following commutative diagram:

\[
\begin{tikzcd}
\Sp_D \arrow[rr] \arrow[dr] && \ba \arrow[dl] \arrow[ll] \\
& \Stone \arrow[ul] \arrow[ur] &
\end{tikzcd}
\]

Stone duality was generalized to compact Hausdorff spaces by de Vries \cite{deV62}. In de Vries duality, with each compact Hausdorff space $X$, we associate the complete boolean algebra $\RO(X)$ of regular open subsets of $X$ equipped with the proximity relation $\prec$ given by $U\prec V$ iff $\cl(U)\subseteq V$. The resulting structures are known as de Vries algebras \cite{Bez10}. Stone duality then lifts to a dual equivalence between the categories $\KHaus$ of compact Hausdorff spaces and $\dev$ of de Vries algebras. 

The same way $\Clop(X)$ can be identified with the idempotents of the $\mathbb R$-algebra $C(X)$ of continuous real-valued functions, $\RO(X)$ can be identified with the idempotents of the $\mathbb R$-algebra $N(X)$ of normal real-valued functions (see, e.g., \cite[Lem.~6.5]{BCM21d}). The notion of a normal function originates in the work of Dilworth \cite{Dil50}, where the MacNeille completion of the lattice $C(X)$ was characterized as the lattice $N(X)$. 

Dilworth's notion of a normal real-valued function requires working with the underlying order of $\mathbb R$. In \cite{BMMO15b}, the notion of a finitely-valued normal function was generalized to an arbitrary totally ordered algebra $D$. This paved a way to generalize boolean powers of $D$, which are defined over a Stone space, to the more general notion of
de Vries powers of $D$, which are defined over a compact Hausdorff space. More precisely, if $(B,\prec)$ is a de Vries algebra and $X$ is its dual compact Hausdorff space, then the de Vries power of $D$ by $(B,\prec)$ is the algebra $FN(X)$ of finitely-valued normal functions $f:X\to D$. The de Vries algebra $(B,\prec)$ can then be identified with the idempotents of $FN(X)$, and the proximity $\prec$ can be lifted to a proximity $\lhd$ on $FN(X)$. Thus, the de Vries power of $D$ by $(B,\prec)$ is the proximity algebra $(FN(X),\lhd)$. 

As was shown in \cite{BMMO15b}, in the special case when $D$ is a totally ordered domain, de Vries powers can be characterized {using the notion of Baer-Specker $D$-algebras. 
These algebras have a long history. We refer to \cite{BMO20e} for details.
In \cite{BMMO15b} de Vries proximities on boolean algebras were generalized to proximity relations on Specker $D$-algebras which has resulted in the category
$\PBSp_D$ of proximity Baer-Specker $D$-algebras. One of the main results of \cite{BMMO15b} establishes that this category is dually equivalent to $\KHaus$, and hence is equivalent to $\dev$, thus yielding the following commutative diagram, which lifts the diagram given above for Stone duality via boolean powers  to de Vries duality via de Vries powers.
\[
\begin{tikzcd}
\PBSp_D \arrow[rr] \arrow[dr] && \dev \arrow[ll] \arrow[dl] \\
& \KHaus \arrow[ur] \arrow[ul] &
\end{tikzcd}
\]

The equivalence between $\dev$ and $\PBSp_D$ is obtained by passing through $\KHaus$, hence is not choice-free. In this article we give a purely algebraic proof of this equivalence, which is choice-free. 
This we do by going back to the original definition of boolean powers by Foster \cite{Fos53a, Fos53b,Fos61}. Using this approach, we can see that Specker $D$-algebras, defined as  idempotent-generated torsion-free $D$-algebras, 
 are boolean powers by utilizing orthogonal decompositions of elements of these algebras  (see Section~\ref{sec:specker}). However, orthogonal decompositions are ill suited for 
  working with a proximity on a Specker $D$-algebra, and so we introduce a different decomposition for the elements in the algebra, a decreasing decomposition  that is reminiscent of Mundici's good sequences (\cite[p.~28]{Mun86}). These decreasing decompositions are the key ingredient in lifting de Vries proximities to proximities on Baer-Specker algebras in a choice-free manner. 

The paper is organized as follows. In Section~\ref{sec:specker} we use the original definition of Foster to give a choice-free proof that boolean powers of a domain $D$ are exactly the Specker $D$-algebras. Starting from Section~\ref{sec:specker over totally ordered}, we assume that $D$ is a totally ordered domain. In Section~\ref{sec:specker over totally ordered} we give a choice-free proof that there is a unique partial ordering on a Specker $D$-algebra $S$ making it a torsion-free $f$-algebra over $D$. In Section~\ref{sec:proximities on Specker} we recall the notion of a proximity on a Specker $D$-algebra and give an alternate description of a boolean power of $D$ using decreasing decompositions. In Section~\ref{sec:lifting} we use decreasing decompositions to lift a de Vries proximity from the boolean algebra $\Id(S)$ of idempotents to the Specker $D$-algebra $S$. This allows us to give a choice-free definition of a de Vries power of $D$ and prove that de Vries powers of $D$ are exactly the Baer-Specker $D$-algebras. Finally, in Section~\ref{sec:7} we give a direct choice-free proof that the category $\dev$ of de Vries algebras is equivalent to the category $\PBSp_D$ of proximity Baer-Specker $D$-algebras.

\section{Specker algebras and boolean powers}\label{sec:specker}

In this section $D$ is an arbitrary fixed integral domain. For a commutative unital $D$-algebra $S$, we let $\Id(S)$ be the boolean algebra of its idempotents.

\begin{definition}
Let $S$ be a commutative unital $D$-algebra. 
\begin{enumerate}
\item We call $S$ {\em idempotent-generated} if $S$ is generated as a $D$-algebra by $\Id(S)$.
\item We call $S$ a {\em Specker $D$-algebra} if $S$ is idempotent-generated and torsion-free as a $D$-module.
\item We denote by $\mathbf{Sp}_D$  the category of Specker $D$-algebras and unital $D$-algebra homomorphisms.
\end{enumerate}
\end{definition}

Various characterizations of Specker $D$-algebras can be found in \cite{BMMO15a}. Some of these we collect in the next theorem. Recall that the boolean power of $D$ by a boolean algebra $B$ is the $D$-algebra $C(X,D)$ of continuous functions $f:X\to D$, where $X$ is the Stone space of $B$ and $D$ is given the discrete topology.

\begin{theorem}\label{thm:Specker}
For a commutative unital $D$-algebra $S$, the following are equivalent:
\begin{enumerate}[$(1)$]
\item $S$ is a Specker $D$-algebra.
\item $S$ is isomorphic to an idempotent-generated subalgebra of a power of $D$.
\item $S$ is isomorphic to a boolean power of $D$.
\item $S$ is idempotent-generated and a free $D$-module.
\end{enumerate}
\end{theorem} 

The definition of a boolean power given before the theorem is due to J\'onsson (see \cite[p.~5]{BN80}). Because the definition involves the Stone space of $B$, it is not choice-free. Since this definition was used in \cite{BMMO15a}, the proof of Theorem~\ref{thm:Specker} is also not choice-free. To avoid this reliance on the axiom of choice, we revert to the  original definition of a boolean power given by Foster (see, e.g., \cite[p.~31]{Fos61}):

\begin{definition}\label{def:boolean power}
The (bounded) {\em boolean power of $D$ by $B$} is the $D$-algebra $D[B]^*$ of finitely-valued functions $f:D \rightarrow B$ such that $f(a) \wedge f(b) = 0$ for all $a \ne b$ in $D$ and $\bigvee {\func{Im}} f = 1$. The algebra operations on $D[B]^*$ are defined as follows, where $a,b\in D$ and $f,g\in D[B]^*$:
\begin{itemize}
\item $(f + g)(a) = \displaystyle{\bigvee} \{{f}(b) \wedge {g}(c): {b+c=a}\}$.
\item ${(fg)}(a) = \displaystyle{\bigvee} \{{f}(b) \wedge {g}(c):{bc=a}\}$.
\item ${(bf)}(a) = \displaystyle{\bigvee} \{f(c):bc=a\}$.
\end{itemize}
\end{definition}

Using Definition~\ref{def:boolean power}, a choice-free proof that $S$ is a Specker $D$-algebra iff $S$ is isomorphic to a boolean power was outlined in \cite[Rem.~2.9]{BMMO15a}. Since this observation is important to our point of view in the present article, we give the details below. For this we  recall orthogonal decompositions of elements of Specker $D$-algebras. 

\begin{definition}
Let $S$ be a Specker $D$-algebra and $B=\Id(S)$. An \emph{orthogonal decomposition} of $s \in S$ is a representation $s = \sum_{i=0}^n a_i e_i$ with $a_i \in D$ (not necessarily distinct) and $e_i \in B$ are pairwise orthogonal (that is, $e_i \wedge e_j = 0$ for each $i \ne j$). If, in addition, $e_0 \vee \cdots \vee e_n = 1$, we call this a \emph{full orthogonal decomposition}.  
\end{definition}

By \cite[Lem.~2.1]{BMMO15a}, each $s \in S$ has a unique full orthogonal decomposition with distinct coefficients.
To connect orthogonal decompositions with the boolean power of $D$ by $B$, let $s=\sum_{i=0}^n a_ie_i$ be a full orthogonal decomposition of $s\in S$ with the $a_i$ distinct. Define $s^\perp:D \rightarrow B$ by
\[
s^\perp(a) =
\left\{\begin{array}{ll}
e_i & \textrm{if } a=a_i \textrm{ for some } i, \\
0 & \textrm{otherwise.} \\
\end{array}\right.
\]
It is straightforward to see that $s^\perp\in D[B]^*$ and $s = \sum_{a\in D} a s^\perp(a)$.
Conversely, if $f \in D[B]^*$, then $s = \sum_{a \in D} af(a)$ is a full orthogonal decomposition of $s$ with distinct coefficients such that $s^\perp = f$. We show that this correspondence is an isomorphism.

\begin{theorem} \label{perp is specker}
Let $S$ be a Specker $D$-algebra and $B=\Id(S)$. Then the map $(-)^\perp:S\to D[B]^*$ is a $D$-algebra isomorphism. 
Moreover, the restriction of $(-)^\perp$ to $B$ is a boolean isomorphism from $B$ to $\Id(D[B]^*)$. 
\end{theorem}

\begin{proof}
We just saw that $(-)^\perp$ is a bijection. It remains to show that $(-)^\perp$ is a $D$-algebra homomorphism. Let $s,t \in S$ and let $s = \sum_i a_i e_i$ and $t = \sum_j b_j f_j$ be full orthogonal decompositions. Then $s = \sum_{i,j} a_i (e_i \wedge f_j)$ and $t = \sum_{i,j}b_j (e_i \wedge f_j)$. Therefore, $s,t$ have full orthogonal decompositions with the same set of idempotents (but not necessarily distinct coefficients). Thus, without loss of generality we may assume that $s$ and $t$ have orthogonal decompositions $s = \sum_i a_i e_i$ and $t = \sum_i b_i e_i$. Then $s^\perp(a) = \bigvee\{ e_i : a = a_i\}$, and a similar description holds for $t^\perp$. Applying the definition of $s^\perp+t^\perp$ and the fact that the $e_i$ are pairwise orthogonal, we obtain
\begin{align*}
(s^\perp+t^\perp)(a) &= \bigvee \left\{s^\perp(b) \wedge t^\perp(c): b+c=a\right\} \\
&= \bigvee \left\{ \bigvee \{e_i: b=a_i\} \wedge \bigvee \{e_j: c=b_j\}: b+c=a \right\} \\
&= \bigvee \left\{ \bigvee \{ e_i \wedge e_j : b=a_i \mbox{ and } c = b_j \} : b+c=a \right\} \\
&= \bigvee \left\{ \bigvee \{ e_i : b=a_i \mbox{ and } c=b_i \} : b+c=a\right\} \\
&= \bigvee \left\{ e_i : a_i+b_i = a\right\} \\
&= (s+t)^\perp(a),
\end{align*}
where the last equality holds since $s+t = \sum_i (a_i + b_i)e_i$. Therefore, $(s+t)^\perp = s^\perp + t^\perp$. The proofs that $(st)^\perp = s^\perp\: t^\perp$ and $(bs)^\perp = bs^\perp$ for $b \in D$ are similar.
Thus, $(-)^\perp$ is a $D$-algebra isomorphism.
Finally, since $(-)^\perp$ is a ring isomorphism from $S$ to $D[B]^*$, it restricts to a boolean isomorphism between $B = \Id(S)$ and $\Id(D[B]^*)$.
\end{proof}

It follows from Theorem~\ref{perp is specker} that if $S$ is a Specker $D$-algebra, then $S$ is isomorphic to the boolean power of $D$ by $\Id(S)$. To prove that every boolean power of $D$ is a Specker $D$-algebra, we require the following construction that has its roots in the work of Bergman \cite{Ber72} and Rota \cite{Rot73}.

\begin{definition} \cite[Def.~2.4]{BMMO15a} 
For a boolean algebra $B$, let $D[B]$ be the quotient ring of the polynomial ring $D[\{x_e : e\in B\}]$ over $D$ in variables indexed by the elements of $B$ modulo the ideal $I_B$ generated by the following elements, as $e,f$ range over $B$: $$x_{e\wedge f} - x_e x_f, \ \ x_{e\vee f} - (x_e + x_f - x_e x_f), \ \ x_{\lnot e} - (1-x_e), \ \ x_0.$$
\end{definition}

The following result follows from the definition of $D[B]$ and \cite[Lem.~3.2(4)]{BMMO15a}.

\begin{theorem} \label{thm: specker}
$D[B]$ is a Specker $D$-algebra and $B$ is isomorphic to $\Id(D[B])$. 
\end{theorem}

We thus are ready for a choice-free proof that Specker $D$-algebras are boolean powers of~$D$.

\begin{corollary} \label{rem:Specker=boolean power} 
Boolean powers of a domain $D$ are, up to isomorphism in the category of $D$-algebras, precisely the Specker $D$-algebras.
\end{corollary}

\begin{proof}
If $S$ is a Specker $D$-algebra, then Theorem~\ref{perp is specker} yields that $S$ is isomorphic to the boolean power $D[\Id(S)]^*$. Conversely, 
by Theorem~\ref{thm: specker}, $D[B]$ is a Specker $D$-algebra and $B$ is isomorphic to $\Id(D[B])$. Thus, by Theorem~\ref{perp is specker}, 
the boolean power of $D$ by $B$ is isomorphic to the Specker $D$-algebra $D[B]$.
\end{proof}

\begin{remark} \label{rem: idempotents of D[B]*}
Let $B \in \ba$ and let $i_B : B \to \Id(D[B])$ be the boolean isomorphism of Theorem~\ref{thm: specker}.
Composing $i_B$ with the boolean isomorphism $(-)^\perp : \Id(D[B]) \to \Id(D[B]^*)$
of Theorem~\ref{perp is specker} yields a boolean isomorphism which sends $e \in B$ to 
$e^\perp \in \Id(D[B]^*)$, given by
\[
e^\perp(a) = \left\{ \begin{array}{ll} e & \textrm{if } a = 1, \\ \lnot e & \textrm{if } a = 0, \\ 0 & \textrm{if } a \ne 0,1. \end{array} \right.
\]
\end{remark}

\begin{remark}
More generally, Specker algebras can be defined over an arbitrary commutative ring $R$ with~1, but the definition is more subtle when zero divisors are present. This was done in \cite{BMMO15a}, where the notion of a faithful generating algebra of idempotents was introduced. If we replace $\Id(S)$ by such a generating algebra, then the proofs of Theorem~\ref{perp is specker} and Corollary~\ref{rem:Specker=boolean power} generalize, thus yielding a choice-free proof of \cite[Thm.~2.7]{BMMO15a} that boolean powers of $R$ are precisely the Specker $R$-algebras. The reason we restrict to domains will become clear when we introduce proximities on Specker algebras; see Remark~\ref{rem: prox def}(4).
\end{remark}

\section{Specker algebras over totally ordered domains}\label{sec:specker over totally ordered} 

From now on we assume that $D$ is a totally ordered domain. It was shown in \cite[Thm.~5.1]{BMMO15a} that there is a unique ordering on a Specker $D$-algebra $S$ that makes $S$ into a torsion-free $f$-algebra over $D$. But the proof is not choice-free. In Theorem~\ref{order prop} we give a choice-free proof of this result, and also show that the isomorphism of Theorem~\ref{perp is specker} is an order isomorphism.

We start by recalling some basic definitions of ordered rings (see, e.g., \cite[Ch.~XVII.5]{Bir79}). A ring $R$ with a partial ordering $\le$ is an \emph{$\ell$-ring} (lattice-ordered ring) if 
\begin{itemize}
\item[(i)] $(R,\le)$ is a lattice;
\item[(ii)] $s \le t$ implies $s + r \le t + r$ for each $r$;
\item[(iii)] $0 \le s, t$ implies $0 \le st$. 
\end{itemize}
An $\ell$-ring $R$ is an \emph{$f$-ring} if for each $r, s, t \in R$ with $s \wedge t = 0$ and $r \ge 0$, we have $rs \wedge t = 0$.

\begin{definition}
Let $(S,\le)$ be a partially ordered $D$-algebra. 
\begin{enumerate}
\item We call $S$ an \emph{$\ell$-algebra over $D$} if $S$ is both an $\ell$-ring and a $D$-algebra such that whenever $0\le s \in S$ and $0 \le a \in D$, then $as \ge 0$. 
\item We call $S$ an \emph{$f$-algebra over $D$} if $S$ is both an $\ell$-algebra over $D$ and an $f$-ring. 
\end{enumerate}
\end{definition}

\begin{theorem} \label{order prop}
Let $S$ be a Specker $D$-algebra. Then there is a unique partial ordering $\le$ on $S$ for which $(S,\le)$ is an $f$-algebra over $D$, given by $s \le t$ if $t-s$ has an  orthogonal decomposition whose coefficients are nonnegative. Moreover, $\le$ restricts to the usual order on $\Id(S)$.
\end{theorem}

\begin{proof} Let $P$ be the set of elements in $S$ that have an orthogonal decomposition whose coefficients are  nonnegative. We prove that $P \cap -P = \{0\}$ and $P$ is closed under addition, multiplication, and multiplication by positive scalars.
Let $s, t \in P$ and let $s = \sum_i a_i e_i$ and $t = \sum_j b_j f_j$ be orthogonal decompositions with $0 \le a_i, b_j$ for each $i, j$. As in the proof of Theorem~\ref{perp is specker}, we may write $s = \sum_{i,j} a_i (e_i \wedge f_j)$ and $t = \sum_{i,j}b_j (e_i \wedge f_j)$. Therefore, $s+t = \sum_{i,j} (a_i + b_j)(e_i \wedge f_j)$ and  $st = \sum_{i,j} a_ib_j (e_i \wedge f_j)$, so $s+t, st \in P$. Moreover, it is clear that $as \in P$ for each $0 \le a \in D$. 
To see that $P \cap -P = \{0\}$, suppose that $s = \sum_i a_i e_i = \sum_j -b_j f_j$ are orthogonal decompositions with each $a_i, b_j \ge 0$. Then $s = \sum_{i,j} a_i (e_i \wedge f_j) = \sum_{i,j} -b_j (e_i\wedge f_j)$, so $0 = \sum_{i,j}(a_i + b_j)(e_i \wedge f_j)$. Multiplying by $e_i \wedge f_j$ yields $(a_i + b_j)(e_i \wedge f_j) = 0$. Since $S$ is a torsion free $D$-module, $a_i + b_j = 0$ or $e_i \wedge f_j = 0$. If $e_i \wedge f_j = 0$, then $a_i (e_i \wedge f_j)=0$. Otherwise $a_i = 0 = b_j$ since both are nonnegative. In either case, $a_i (e_i \wedge f_j) = 0$ for each $i,j$, and so $s = 0$.
Thus, if we set $s \le t$ whenever $t-s \in P$, then $\le$ is a partial ordering on $S$ and $(S,\le)$ satisfies conditions (ii) and (iii) of the definition of an $\ell$-ring (see \cite[Thm.~VI.1.1]{Fuc63}).

To see that $(S,\le)$ also satisfies (i), let $s,t \in S$ and let $s = \sum_i a_i e_i$ and $t = \sum_i b_i e_i$ be orthogonal decompositions of $s$ and $t$ with the same set of idempotents. The join and meet of $s, t$ exist and
are given by:

\begin{claim}  \label{claim: max and min}
$s \vee t = \sum_i \max(a_i, b_i) e_i$ and $s \wedge t = \sum_i \min(a_i, b_i) e_i$. 
\end{claim}

\begin{proof}[Proof of Claim:]
The proofs of the two parts of the claim are similar, so we only prove the second. 
Set $r = \sum_i \min(a_i, b_i) e_i$. The definition of $P$ shows that $r \le s, t$. Next, let $q \in S$ be a lower bound of $s, t$. 
By refining the decompositions and eliminating zero idempotents if necessary,
we may assume that $q = \sum_i d_i e_i$ for some $d_i \in D$ and that all $e_i \ne 0$. Since $q \le s$, the $e_i$ are pairwise orthogonal, and $e_i \ge 0$, we have that $d_i e_i = qe_i \le se_i = a_ie_i$, so $(a_i - d_i)e_i \in P$. If $a_i < d_i $, then $(d_i - a_i)e_i \in P$. 
Because $P \cap -P = \{0\}$, this forces $(a_i - d_i)e_i = 0$, so $a_i = d_i$ since $S$ is torsion-free over $D$. This contradiction shows that $d_i \le a_i$. Similarly, $d_i \le b_i$, so $d_i \le \min(a_i, b_i)$. Therefore, $q \le r$. Thus, $r$ is the greatest lower bound of $s, t$, and so $s \wedge t$ exists in $S$ and is equal to $\sum_i \min(a_i, b_i) e_i$.
\end{proof}

Consequently, $S$ is an $\ell$-ring. 
That $0\le s \in S$ and $0 \le a \in D$ imply $as \ge 0$ is easy to see. Thus, $S$ is an $\ell$-algebra over $D$. 
To see that $S$ is an $f$-algebra, let $s\wedge t=0$ and $r\in S$ with $r\ge 0$. As above, $s$, $t$, and $r$ have orthogonal decompositions $s = \sum_i a_i e_i$, $t = \sum_i b_i e_i$, and $r = \sum_i c_i e_i$ with the same set of idempotents and $0 \le a_i, b_i, c_i$, and we may assume without loss of generality that each $e_i \ne 0$. By the claim, $s\wedge t=\sum_i \min(a_i, b_i) e_i$. Since $s \wedge t = 0$, for each $i$, either $a_i = 0$ or $b_i = 0$. Because $sr\wedge t=\sum_i \min(a_ic_i,b_i) e_i$, we see that $sr \wedge t = 0$. Consequently, $S$ is an $f$-algebra over $D$. This in particular implies that the order on $S$ restricts to the usual order on $\Id(S)$ (see \cite[Lem~4.9(2)]{BMMO15b}).
Finally, the proof of uniqueness of $\le$ is a direct adaptation of that given in \cite[Thm.~5.1]{BMMO15a}. 
\end{proof}

It was proved in \cite[Cor.~5.3]{BMMO15a} that each unital $D$-algebra homomorphism between Specker $D$-algebras is an $\ell$-algebra homomorphism. The proof used \cite[Thm.~5.1]{BMMO15a} and hence was not choice-free. By using the original argument from \cite{BMMO15a} but substituting the choice-free  Theorem~\ref{order prop} for that of choice-dependent Theorem 5.1 from \cite{BMMO15a}, we therefore obtain 
 a choice-free proof of this result.

\begin{theorem} \label{thm: morphisms are l-algebra morphisms}
Each unital $D$-algebra homomorphism between Specker $D$-algebras is an $\ell$-algebra homomorphism.
\end{theorem}

As an immediate consequence of Theorem~\ref{thm: morphisms are l-algebra morphisms} we obtain:

\begin{corollary}
Let $S$ be a Specker $D$-algebra and $B=\Id(S)$. The map $(-)^\perp : S \to D[B]^*$ of Theorem~\emph{\ref{perp is specker}} is an $\ell$-algebra isomorphism.
\end{corollary}

\begin{remark} \label{order remark}
Let $D$ be a totally ordered domain and $B$ a boolean algebra. 
Clearly $D$ is a lattice, where $a\wedge b=\min(a,b)$ and $a\vee b=\max(a,b)$ for each $a,b\in D$.
\begin{enumerate}[$(1)$]
\item The positive cone $P$ of $D[B]^*$ for the partial order $\le$ defined in Theorem~\ref{order prop} can be described by 
\[
f \in P \mbox{ iff } f(a) = 0 \mbox{ for each } a < 0.
\] 
To see this, if $f \in D[B]^*$, then by the comments before Theorem~\ref{perp is specker} applied to the map $(-)^\perp : D[B] \to D[B]^*$, we have $f = (\sum_{a \in D} a f(a))^\perp$.
If $f(a) = 0$ for each $a < 0$, then the description of $P$ in the proof of Theorem~\ref{order prop} shows that $f \in P$. Conversely, if $f \in P$, then there are $e_i \in B$ and $0 \le a_i \in D$ with $f = \sum_{i=1}^n a_i e_i^\perp$. Let $a < 0$. Then
\[
f(a) = \bigvee \{ a_1 e_1^\perp(b_1) \wedge \cdots \wedge a_ne_n^\perp(b_n) : a_1b_1 + \cdots + a_nb_n = a\}.
\]
Because $a < 0$ and all $a_i \ge 0$, if $\sum_i a_ib_i =a$, then some $b_i < 0$, and so $e_i^\perp(b_i) = 0$. This shows that $f(a) = 0$. From this we see that 
\[
f \le g \mbox{ iff } (g-f)(a) = 0 \mbox{ for each } a < 0.
\]

\item
The meet and join in $D[B]^*$ are calculated by
\begin{align*}
(f \wedge g)(a) &= {\displaystyle{\bigvee}} \{f(b) \wedge g(c): {\min(b,c) = a}\}, \\
(f \vee g)(a) &= {\displaystyle{\bigvee}} \{f(b) \wedge g(c): {\max(b,c) = a}\}.
\end{align*}
We only prove the first equality as the second is proved similarly. Define $h : D \to B$ by $h(a) = {\displaystyle{\bigvee}} \{f(b) \wedge g(c): {\min(b,c) = a}\}$. We show that $h=f\wedge g$. It is easy to see 
that $h \in D[B]^*$. Therefore, $-h(c) = h(-c)$ for each $c \in D$ by the definition of scalar multiplication in $D[B]^*$.
To see that $h \le f,g$, let $a < 0$. Then 
\begin{align*}
(f - h)(a) &= \bigvee \{ f(b) \wedge (-h)(c) : b+c = a\} = \bigvee \{ f(b) \wedge h(-c) : b+c = a\} \\
&= \bigvee \{ f(b) \wedge \bigvee \{ f(d) \wedge g(e) : \min(d, e) = -c\} : b+c = a\} \\
&= \bigvee \{ f(b) \wedge f(d) \wedge g(e) : \min(d, e) = -c, b+c = a\}.
\end{align*}
If $b + c = a$, then $b < -c$ since $a < 0$. Therefore, if $\min(d, e) = -c$, then $b < d$. This implies that $f(b) \wedge f(d) = 0$ and hence $(f-h)(a) = 0$. Thus, $h \le f$. Similarly, $h \le g$, which gives $h \le f \wedge g$. 

To see the reverse inequality, suppose that $k \in D[B]^*$ with $k \le f, g$.  We show that $k \le h$. It follows from (1) that $k \le f$ implies $(f-k)(a) = 0$ for all $a < 0$. As we saw above,
\[
(f-k)(a) = \bigvee \{ f(b) \wedge k(-c) : b + c = a\}.
\]
Therefore, $f(b) \wedge k(-c) = 0$ whenever $b + c < 0$. Similarly, $g(b) \wedge k(-c) = 0$ whenever $b + c < 0$. We have
\begin{align*}
(h-k)(a) &= \bigvee \{ h(b) \wedge k(-c) : b + c = a\} \\
&= \bigvee \{ \bigvee \{ f(b_1) \wedge g(b_2) :\min(b_1,b_2) = b \} \wedge k(-c) : b + c = a\} \\
&= \bigvee \{ f(b_1) \wedge g(b_2) \wedge k(-c) : \min(b_1, b_2) = b, b + c = a\}.
\end{align*}
If $\min(b_1, b_2) = b$ and $b + c = a < 0$, then either $b_1 + c < 0$ or $b_2 + c < 0$. Therefore, either $f(b_1) \wedge k(-c) = 0$ or $f(b_2) \wedge k(-c) = 0$. Consequently, if $\min(b_1, b_2) + c = a$, then $f(b_1) \wedge g(b_2) \wedge k(-c) = 0$. From this it follows that $(h-k)(a) = 0$ if $a < 0$, so $k \le h$. In particular, $f \wedge g \le h$. Thus, $h = f \wedge g$.
\end{enumerate}
\end{remark}

\begin{remark}
Let $S$ be a Specker $D$-algebra and $s, t \in S$. It is not true in general that $s \leq t$ iff $s^\perp(a) \leq t^\perp(a)$ for all $a \in D$. For example, while $0 \le 1$, we have $1 = 0^\perp(0)  \not\leq 1^\perp(0) = 0$ because the full orthogonal decompositions of $0, 1$ are $0 = 0 \cdot 1$ and $1 = 1 \cdot 1$, respectively. This drawback will be corrected in Lemma~\ref{* arithmetic}(2) using a different way of viewing boolean powers of $D$, which we turn to next.
\end{remark}

\section{Proximities on Specker algebras and decreasing decompositions} \label{sec:proximities on Specker}

As we pointed out in the introduction, for a compact Hausdorff space $X$, there is a standard notion of proximity on the boolean algebra $\RO(X)$ of regular open subsets of $X$ given by $U\prec V$ iff $\cl(U)\subseteq V$. De Vries \cite{deV62} axiomatized proximity relations on arbitrary boolean algebras. This has resulted in the notion of a de Vries proximity $\prec$ on a boolean algebra $B$. A {\em de Vries algebra} is a pair $(B,\prec)$, where $B$ is a complete boolean algebra and $\prec$ is a de Vries proximity on $B$ (see \cite{Bez10}). 

In \cite{BMMO15b} de Vries proximities on boolean algebras were generalized to proximities on arbitrary torsion-free $f$-algebras over a totally ordered domain $D$. 

\begin{definition} \label{proximity definition} \cite[Def.~4.2]{BMMO15b}
Let $S$ be a torsion-free $f$-algebra over $D$. We call a binary relation $\lhd$ on $S$ a {\em proximity} if the following axioms are satisfied:
\begin{enumerate}
\item[(P1)] $0\lhd 0$ and $1 \lhd 1$.
\item[(P2)] $s \lhd t$ implies $s \le t$.
\item[(P3)] $s \le t \lhd r \le u$ implies $s \lhd u$.
\item[(P4)] $s \lhd t,r$ implies $s \lhd t \wedge r$.
\item[(P5)] $s \lhd t$ implies $-t \lhd -s$.
\item[(P6)] $s \lhd t$ and $r \lhd u$ imply $s+r \lhd t+u$.
\item[(P7)] $s \lhd t$ implies $as \lhd at$ for each $0 < a \in D$, and $as \lhd at$ for some $0 < a \in D$ implies $s \lhd t$.
\item[(P8)] $s,t,r,u \ge 0$ with $s \lhd t$ and $r \lhd u$ imply $sr \lhd tu$.
\item[(P9)] $s \lhd t$ implies there is $r\in S$ with $s \lhd r \lhd t$.
\item[(P10)] $s > 0$ implies there is $0 < t\in S$ with $t \lhd s$.
\end{enumerate}
We call a pair $(S,\lhd)$ a {\em proximity $D$-algebra} if $S$ is a torsion-free $f$-algebra over $D$ and $\lhd$ is a proximity on $S$. If $S$ is a Specker $D$-algebra, then we call $(S,\lhd)$ a {\em proximity Specker $D$-algebra}.
\end{definition}

\begin{remark} \label{rem: prox def}
\begin{enumerate}
\item[]
\item The axioms (P1)--(P5) and (P9)--(P10) are direct analogues of the corresponding de Vries axioms, while the axioms (P6)--(P8) govern the interaction between the algebra operations and proximity on $S$.
\item Since every Specker $D$-algebra $S$ is torsion-free, if $S$ is nonzero, we always identify $D$ with a subalgebra of $S$ by sending $a \in D$ to $a\cdot 1 \in S$.
\item It is an easy consequence of the axioms that $s\lhd t$ and $r\lhd u$ imply $s\wedge r\lhd t\wedge u$ and $s\vee r\lhd t\vee u$, and that $s \lhd t$ iff $as \lhd bt$ for $0 < a\le b \in D$. Hence, it follows from (P1), (P7), and (P5) that for each $a \in D$, we have $a \lhd a$.
\item The right-to-left implication in (P7) plays an important role in our considerations (see the proofs of Propositions~\ref{prop:4.9} and \ref{prop:5.4}). This implication is problematic if $D$ is not a domain, so 
in Definition~\ref{proximity definition} it is essential that $D$ is a domain.
\end{enumerate}
\end{remark}

Let $S$ be a Specker $D$-algebra and $B=\Id(S)$ the boolean algebra of idempotents of $S$. If $\lhd$ is a proximity on $S$, we can consider its restriction to $B$. It was shown in \cite{BMMO15b} using orthogonal decompositions that the restriction of $\lhd$ is a de Vries proximity on $B$. The proof in \cite{BMMO15b} is choice-free.

\begin{proposition} \label{prop:4.6} \cite[Prop.~5.1]{BMMO15b}
Let $\lhd$ be a proximity on a Specker $D$-algebra $S$. Then $\lhd$ restricts to a de Vries proximity on $\Id(S)$.
\end{proposition}

Our next goal is to prove the converse of Proposition~\ref{prop:4.6}, that a de Vries proximity on $\Id(S)$ has a unique extension to a proximity on $S$.
For this we need to work with decreasing decompositions instead of orthogonal decompositions. Decreasing decompositions, which are similar to Mundici's good sequences \cite[p.~28]{Mun86}, were studied for Specker algebras in \cite[Sec.~5]{BMMO15b}, where it was shown how to go back and forth between orthogonal and decreasing decompositions. 

\begin{remark} \label{rem: from orthogonal to decreasing}
We will briefly describe how to go from an orthogonal decomposition to a decreasing decomposition since this will be used in Proposition~\ref{prop: relation between flat and perp}. Let $S$ be a Specker $D$-algebra, $s \in S$, and $s=\sum_{i=0}^n a_i f_i$ be an orthogonal decomposition of $s$ with the $a_i \in D$ distinct and nonzero. Without loss of generality we may assume that $a_0 < \cdots < a_n$. We can then write
\[
s = a_0 (f_0 + \cdots + f_n) + (a_1-a_0)(f_1 + \cdots + f_n) + \cdots + (a_n - a_{n-1})f_n.
 \]
Therefore, $s = \sum_{i=0}^n b_ie_i$, where $b_0 = a_0$,  $b_i=a_i-a_{i-1}$ for $i\ge 1$, and $e_i=\sum_{j=0}^i f_j=\bigvee_{j=0}^i f_j$, where the second equality follows from \cite[Eqn.~XIII.3(14)]{Bir79}. This exhibits $s$ as a linear combination of a sequence of strictly decreasing idempotents. Moreover, all the coefficients are nonzero and all of them except possibly $b_0$ are positive. Furthermore, if $s=\sum_{i=0}^n a_i f_i$ is a full orthogonal decomposition of $s$, then $e_0 = 1$. In this case we will write the corresponding decreasing decomposition as $s = a_0 + \sum_{i=1}^n b_i e_i$.
\end{remark}

\begin{definition} \label{def: decreasing form}
Let $S$ be a Specker $D$-algebra and let $s\in S$.
\begin{enumerate}[$(1)$]
\item We say that $s$ is in {\em decreasing form} if $s = \sum_{i=0}^n b_ie_i$ with $e_0 \ge \cdots \ge e_{n}$, $b_0 \ne 0$ and $b_i > 0$ for $i\ge 1$.
\item We say that $s$ is in {\em full decreasing form} if $s = a_0 + \sum_{i=1}^n b_ie_i$ is in decreasing form (with $e_0=1$).
\end{enumerate}
\end{definition}

\begin{remark}
\begin{enumerate}
\item[]
\item Because each element of $S$ has a full orthogonal decomposition, each element has a full decreasing decomposition. Moreover, since a full orthogonal decomposition with distinct nonzero coefficients is unique, each $s \in S $ has a unique representation as $s = a_0 + \sum_{i=1}^n b_ie_i$ with each $b_i > 0$ and $1 = e_0 > e_1 > \cdots > e_n$.

\item As we saw in Section~\ref{sec:specker}, to write two elements in compatible orthogonal form, we cannot assume coefficients are distinct. Similarly, we will see in Lemma~\ref{* properties}(2) that two elements have a compatible decreasing decomposition, but we cannot assume that the idempotents are strictly decreasing. It is for this reason that the idempotents in Definition~\ref{def: decreasing form}(1) are not assumed to be strictly decreasing.
\end{enumerate}
\end{remark}

Using decreasing decompositions, we give an alternative view of boolean powers of $D$.

\begin{definition} \label{star definition}
Let $B$ be a boolean algebra. We define $D[B]^\flat$ to be the set of all decreasing functions $f:D\to B$ for which there exist $1=e_0 > e_1 > \cdots > e_n > 0$ in  $B$ and $a_0 < a_1 < \cdots < a_n$ in $D$ such that
\[
f(a) =
\left\{\begin{array}{ll}
1 & \textrm{if } a \le a_0, \\
e_i & \textrm{if } a_{i-1} < a \le a_i, \\
0 & \textrm{if } a_n < a.\\
\end{array}\right.
\]
\end{definition}

Let $S$ be a Specker $D$-algebra and $B=\Id(S)$. The following proposition illustrates that $D[B]^\flat$ encodes decreasing decompositions of elements of $S$ into an algebra of functions from $D$ to $B$.

\begin{proposition} \label{decr dec}
Let $S$ be a Specker $D$-algebra and $B=\Id(S)$.
\begin{enumerate}[$(1)$]
\item Let $s \in S$ be in full decreasing form $s = a_0 + \sum_{i=1}^n b_i e_i$ and set $a_i = a_0 + b_1 + \cdots + b_i$ for $1 \le i \le n$. Define $s^\flat:D \rightarrow B$ by
\[
s^\flat(a) = \left\{\begin{array}{lll}
1 & \textrm{if} & a \le a_0, \\
e_i & \textrm{if} & a_{i-1} < a \le a_i, \\
0 & \textrm{if} & a_n < a.
\end{array}\right.
\]
Then $s^\flat\in D[B]^\flat$.
\item Conversely, for $f \in D[B]^\flat$, let the image of $f$ in $B$ be $ \{1=e_0 > e_1 > \cdots > e_n > 0\}$, and for each $i \leq n$, let $a_i$ be the largest element of $f^{-1}(e_i)$. Then $s = a_0 + \sum_{i=1}^n (a_i - a_{i-1})f(a_i)$ is an element of $S$ in full decreasing form and $s^\flat = f$.
\end{enumerate}
\end{proposition}

\begin{proof}
Straightforward.
\end{proof}

Let $S$ be a Specker $D$-algebra and $B = \Id(S)$. The reader probably already anticipates that $D[B]^\flat$ is a $D$-algebra and that $(-)^\flat:S\to D[B]^\flat$ is a $D$-algebra isomorphism. This in particular implies that $D[B]^*$ and $D[B]^\flat$ provide two alternative representations of $S$, one that encodes orthogonal decompositions and the other that encodes decreasing decompositions. Consequently, $D[B]^\flat$ also provides an alternative way to view boolean powers of a totally ordered domain $D$.

We first prove that $D[B]^*$ is in bijective correspondence with $D[B]^\flat$, and describe the $D$-algebra structure of $D[B]^\flat$ induced by this bijection. From this we will then derive that $(-)^\flat:S\to D[B]^\flat$ is a $D$-algebra isomorphism.

\begin{theorem} \label{* axiomitization}
For a boolean algebra $B$, there is a bijection between $D[B]^*$ and $D[B]^\flat$ that induces on $D[B]^\flat$ the structure of a Specker $D$-algebra whose operations satisfy,
for all $f,g \in D[B]^\flat$ and $a,b \in D$,
\begin{enumerate}[$(1)$]
\item $(f+g)(a) = \displaystyle{\bigvee} \{f(b_1) \wedge g(b_2):b_1+b_2 \ge a\}$.
\item If $b>0$, then $(bf)(a) = \displaystyle{\bigvee} \{f(c):bc\ge a\}$.
\item If $f,g \geq 0$, then $(fg)(a) =\displaystyle{\bigvee} \{f(b_1) \wedge g(b_2):b_1, b_2 \ge 0, b_1b_2 \ge a\}$.
\end{enumerate}
\end{theorem}

\begin{proof}
Define $\alpha:D[B]^* \rightarrow D[B]^\flat$ by $\alpha(f)(a)=\displaystyle{\bigvee} \{f(b):b \geq a\}$ for each $f \in D[B]^*$ and $a \in D$. To see that $\alpha$ is well defined, let $f \in D[B]^*$, and let $\{a \in D : f(a) \ne 0\} = \{a_0 < \cdots < a_n\}$. Set $e_i = f(a_i) \vee \cdots \vee f(a_n)$ for $0 \le i \le n$. Then $1=e_0 > e_1 > \cdots > e_n > 0$, $\alpha(f)^{-1}(1) = (-\infty, a_0]$, and $\alpha(f)^{-1}(0) = (a_n,\infty)$.  Moreover, if $a_{i-1} < a \le a_i$, then $\alpha(f)(a) = f(a_i) \vee \cdots \vee f(a_n) = e_i$, so that $\alpha(f)^{-1}(e_i) = (a_{i-1},a_{i}]$. Thus, $\alpha(f) \in D[B]^\flat$, and $\alpha$ is well defined.

To see that $\alpha$ is onto, let $g \in D[B]^\flat$, and let $\{1 = e_0 > e_1 > \cdots > e_n > 0\}$ be the image of $g$ in $B$. For each $i$, let $a_i$ be the largest element of $g^{-1}(e_i)$. Define $f:D\to B$ by $f(a_0) = 1$, $f(a_n) = e_n$, $f(a_i) = e_i \wedge \lnot e_{i+1}$ for each $1 \le i \le n-1$, and $f(a) = 0$ for all $a  \in D \setminus \{a_0,\ldots,a_n\}$. Then $f\in D[B]^*$. We show that $\alpha(f) = g$. If $a \le a_0$, then $\alpha(f)(a) = 1$ as it is the join of the $f(b)$ over all $b \ge a$, so $\alpha(f)(a)=g(a)$. If $a_0 < a \le a_1$, then $\alpha(f)(a)$ is the join of $e_1 \wedge \lnot e_2,e_2 \wedge \lnot e_3,\dots,e_{n-1} \wedge \lnot e_n,e_n$, which is $e_1= g(a)$. Similarly, if $a_{i-1} < a \le a_i$, then $\alpha(f)(a) = e_i = g(a)$, and if $a_n < a$, then $\alpha(f)(a)=0=g(a)$. Thus, $\alpha(f) = g$.

To see that $\alpha$ is 1-1, let $f,g \in D[B]^*$ with $\alpha(f) = \alpha(g)$. For $a \in D$ we have
\[
\alpha(f)(a) = f(a) \vee \bigvee \{ f(b) : b > a\}. 
\]
Since the values of $f$ are pairwise orthogonal, $f(a) = \alpha(f)(a) \wedge \lnot \bigvee \{ f(b) : b > a\}$. However, $\bigvee \{ f(b) : b > a\} = \bigvee \{ \alpha(f)(b) : b > a\}$, so $f(a) = \alpha(f)(a) \wedge \lnot \bigvee \{ \alpha(f)(b) : b > a\}$. Similarly, $g(a) = \alpha(g)(a) \wedge \lnot \bigvee \{ \alpha(g)(b) : b > a\}$. Since $\alpha(f) = \alpha(g)$, we see that $f(a) = g(a)$. Thus, $f = g$.

Now since $\alpha : D[B]^* \to D[B]^\flat$ is a bijection, $D[B]^\flat$ inherits the structure of a Specker $D$-algebra from $D[B]^*$. Therefore, what remains to verify is that the algebraic structure that $D[B]^\flat$ inherits from $D[B]^*$ satisfies (1)--(3) of the theorem. In light of the bijection with $D[B]^*$, it suffices to show that if $f,g \in D[B]^*$, then $\alpha(f), \alpha(g)$ both behave as stated in (1)--(3). Thus, we assume that $f,g \in D[B]^*$ and $a,b\in D$. Since the proofs are similar, we only prove (1).

Using the definition of $f+g$ in $D[B]^*$, we have:
\begin{align*}
(\alpha(f) + \alpha(g))(a) &= \bigvee_{b_1+b_2\ge a} \alpha(f)(b_1) \wedge \alpha(g)(b_2) \\
&= \bigvee_{b_1 + b_2 \ge a} \left(\bigvee_{c_1 \ge b_1} f(c_1)\right) \wedge \left (\bigvee_{c_2 \ge b_2} g(c_2)\right) = \bigvee_{b_1 + b_2 \ge a} \left(\bigvee_{c_i \ge b_i} f(c_1) \wedge g(c_2)\right) \\
&= \bigvee_{c \ge a} \left(\bigvee_{c_1 + c_2 = c} f(c_1) \wedge g(c_2)\right) = \bigvee_{c \ge a} {(f+g)}(c) = \alpha(f+g)(a).
\end{align*}
\end{proof}

\begin{remark} \label{rem: idempotents of D[B]flat}
By Remark~\ref{rem: idempotents of D[B]*}, there is an isomorphism $B \to \Id(D[B]^*)$ sending $e$ to $e^\perp$ for each $e \in B$. Since $\alpha : D[B]^* \to D[B]^\flat$ restricts to an isomorphism from $\Id(D[B]^*)$ to $\Id(D[B]^\flat) $, the composition $\tau_B$ is an isomorphism from $B$ to $\Id(D[B]^\flat)$. If $e^\flat = \tau_B(e)$, then it follows from the definition of $\alpha$ and the description of $e^\perp$ that
\[
e^\flat(a) = \left\{ \begin{array}{ll} 1 & \textrm{if } a \le 0,\\  e & \textrm{if } 0 < a \le 1, \\ 0 & \textrm{if } 1 < a.\end{array} \right.
\]
In particular, $0^\flat(a) = 1$ if $a \le 0$ and $0^\flat(a) = 0$ if $0 < a$. Similarly, $1^\flat(a) = 1$ if $a \le 1$ and $1^\flat(a) = 0$ if $1 < a$. We note that $0^\flat$ and $1^\flat$ are then the $0$ and $1$ of $D[B]^\flat$, respectively.
\end{remark}

\begin{proposition} \label{prop: relation between flat and perp}
Let $S$ be a Specker $D$-algebra and $B=\Id(S)$. The following diagram commutes.
\[
\begin{tikzcd}[row sep = 1pc]
& D[B]^\flat \\
S \arrow[ur, "(-)^\flat"] \arrow[dr, "(-)^\perp"'] & \\
& D[B]^* \arrow[uu, "\alpha"']
\end{tikzcd}
\]
Consequently, $(-)^\flat$ is an $\ell$-algebra isomorphism.  
\end{proposition}

\begin{proof}
We first show that $\alpha \circ (-)^\perp = (-)^\flat$. Let $s\in S$ and write $s = \sum_{i=0}^n a_i e_i$ in full orthogonal form with $a_0 < \cdots < a_n$. Then the full decreasing form of $s$ is $a_0 + (a_1-a_0) f_1 + \cdots + (a_n - a_{n-1})f_n$, where $f_i = e_i \vee \cdots \vee e_n$ by Remark~\ref{rem: from orthogonal to decreasing}. Therefore, $s^\flat(a) = 1$ if $a \le a_1$, $s^\flat(a) = f_i$ if $a_{i-1} < a \le a_i$, and $s^\flat(a) = 0$ if $a_n < a$. On the other hand, $\alpha(s^\perp)(a) = \bigvee\{ e_i : a \ge a_i\}$. Thus, $\alpha(s^\perp)(a) = 1$ if $a \le a_0$, $\alpha(s^\perp)(a) = e_i \vee \cdots \vee e_n = f_i$ if $a_{i-1} < a \le a_i$, and $\alpha(s^\perp)(a) = 0$ if $a_n < a$. Consequently, $\alpha(s^\perp) = s^\flat$, and hence the diagram commutes. 

To conclude the proof, it follows from Theorem~\ref{perp is specker} that $(-)^\perp:S\to D[B]^*$ is a $D$-algebra isomorphism, and it follows from Theorem~\ref{* axiomitization} that $\alpha:D[B]^*\to D[B]^\flat$ is a $D$-algebra isomorphism. Therefore, $(-)^\flat$ is a $D$-algebra isomorphism, thus an $\ell$-algebra isomorphism by Theorem~\ref{thm: morphisms are l-algebra morphisms}. 
\end{proof}

\section{De Vries powers} \label{sec:lifting}

In this section we use decreasing decompositions to lift de Vries proximities on boolean algebras to proximities on Specker $D$-algebras. In the particular case in which $(B,\prec)$ is a de Vries algebra, we lift $\prec$ to a proximity $\prec^\flat$ on $D[B]^\flat$ to
 obtain that $(D[B]^\flat,\prec^\flat)$ is a proximity Specker $D$-algebra, which in addition is a Baer ring (defined below). 
 Following \cite{BMMO15b}, we term such algebras proximity Baer-Specker $D$-algebras. The pair $(D[B]^\flat,\prec^\flat)$ provides a choice-free description of the de Vries power of $D$ by $(B,\prec)$ that in \cite[Def.~3.3]{BMMO15b} was defined in a choice-dependent way via the dual compact Hausdorff space of $(B, \prec)$.

We start by showing that the order on $D[B]^\flat$ is pointwise.

\begin{lemma} \label{* arithmetic}
Let $B$ be a boolean algebra. For $f, g \in D[B]^\flat$ we have:
\begin{enumerate}[$(1)$]
\item $(f \wedge g)(a) = f(a) \wedge g(a)$ for each $a \in D$.
\item $f \le g$ iff $f(a) \le g(a)$ for each $a \in D$.
\end{enumerate}
\end{lemma}

\begin{proof}
(1) Since $\alpha : D[B]^* \to D[B]^\flat$ is a bijection, there are $s, t \in D[B]^*$ with $f = \alpha(s)$ and $g = \alpha(t)$. By Remark~\ref{order remark},
$(s \wedge t)(a) = \displaystyle{\bigvee} \{ s(b) \wedge t(c): {\min(b, c) = a}\}$. Therefore, 
\begin{align*}
f(a)\wedge g(a) &= \alpha(s)(a) \wedge \alpha(t)(a) = \left(\bigvee_{b_1 \ge a} s(b_1)\right) \wedge \left( \bigvee_{b_2 \ge a} t(b_2) \right) \\
&= \bigvee_{b_1,b_2 \ge a} s(b_1) \wedge t(b_2) = \bigvee_{b \ge a} \left(\bigvee_{\min(b_1, b_2) = b} s(b_1) \wedge t(b_2)\right) \\
&= \bigvee_{b \ge a} (s\wedge t)(b) = \alpha(s \wedge t)(a) = (f \wedge g)(a).
\end{align*}

(2) We have $f \leq g$ iff $f = f\wedge g$. Therefore, (2) follows from (1).
\end{proof}

Let $S$ be a Specker $D$-algebra and $B = \Id(S)$. Since $(-)^\flat : S \to D[B]^\flat$ is an $\ell$-algebra isomorphism by Proposition~\ref{prop: relation between flat and perp}, the following is an immediate consequence of Lemma~\ref{* arithmetic}.

\begin{lemma} \label{lem: * arithmetic}
Let $S$ be a Specker $D$-algebra. For $s, t \in S$ we have:
\begin{enumerate}[$(1)$]
\item $(s\wedge t)^\flat(a) = s^\flat(a) \wedge t^\flat(a)$ for each $a \in D$. 
\item $s \leq t$ iff $s^\flat(a) \leq t^\flat(a)$ for each $a \in D$.
\end{enumerate}
\end{lemma}

\begin{remark}
In contrast to Lemma~\ref{lem: * arithmetic}(2), as we observed in Remark~\ref{order remark}, it is not the case that $s \leq t$ iff $s^\perp(a) \leq t^\perp(a)$ for all $a \in D$. 
\end{remark}

The next technical lemma is needed in Proposition~\ref{prop:4.9}.

\begin{lemma}\label{* properties}
Let $S$ be a Specker $D$-algebra.
\begin{enumerate}[$(1)$]
\item Let $s \in S$ and let $a,b \in D$ with $a < b$. If
$s^\perp(c) =0$ for all $c$ with 
 $a < c < b$, then $$(s \wedge b) - (s \wedge a) =  [(s-a) \wedge(b-a)] \vee 0 = (b-a)s^\flat(b).$$
\item Let $s,t \in S$. Then there exist $a_0 < \cdots < a_n$ in $D$ with $a_0\le s,t\le a_n$ such that $s$ and $t$ have compatible decreasing decompositions $s = a_0 + \sum_{i=1}^n (a_i-a_{i-1})s^\flat(a_i)$ and $t = a_0 + \sum_{i=1}^n (a_i-a_{i-1})t^\flat(a_i)$. Moreover, if $s,t \ge 0$, then we may assume $a_0 = 0$.
\end{enumerate}
\end{lemma}

\begin{proof}
(1) The proof that $(s \wedge b) - (s \wedge a) =  [(s-a) \wedge(b-a)] \vee 0$ is given in \cite[Claim~6.8]{BMMO15b}.
We show that $(s\wedge b) - (s\wedge a) = (b-a)s^\flat(b)$. As discussed in Section~\ref{sec:specker}, we may write $s = \sum_{b\in D} b s^\perp(b)$. By assumption, $\sum_{a < c < b} c s^\perp(c) =0$, so
\[
s  = \sum_{c \le a} c  s^\perp(c) + \sum_{b \leq c} c  s^\perp(c).
\]
Because $\{ s^\perp(c) : c \in D\}$ is a set of orthogonal idempotents whose join is $1$ and $a,b \in D$, we have $a = \sum_{c \in D} as^\perp(c)$ and $b = \sum_{c \in D} bs^\perp(c)$. Therefore, by Claim~\ref{claim: max and min},
\[
s \wedge b = \sum_{c \le a} \func{min}(b,c)  s^\perp(c) + \sum_{b \leq c} \func{min}(b,c)  s^\perp(c) = \sum_{c \le a} c  s^\perp(c) + \sum_{b \leq c} b  s^\perp(c),
\]
while
\[
s \wedge a = \sum_{c \le a} \func{min}(a,c) s^\perp(c) + \sum_{b \leq c} \func{min}(a,c)  s^\perp(c) = \sum_{c \le a} c  s^\perp(c) + \sum_{b \leq c} a  s^\perp(c).
\]
Thus, as the $s^\perp(c)$ are orthogonal,
\[
(s \wedge b) - (s \wedge a) = (b-a)\sum_{b\le c}s^\perp(c)  = (b - a)  s^\flat(b). 
\]

(2) We first show that for each $s \in S$ there is $0 \le b \in D$ with $-b \le s \le b$. Write $s = \sum_{i=0}^n c_ie_i$ with $a_i \in D$ and $e_i \in \Id(S)$. Let $b_i = \max(c_i, -c_i)$ and $b = \sum_{i=0}^n b_i$. Since $0 \le e_i \le 1$ for each $i$, we have $-b_i \le c_i e_i \le b_i$, so $-b \le s \le b$. Therefore, there are $a_0, a_n \in D$ with $a_0 \le s, t \le a_n$. 

Since $s^\perp, t^\perp$ have finitely many nonzero values,
there are $a_1 ,\dots , a_{n-1}$ in $D$ such that $a_0 < a_1 < \cdots  < a_{n-1} < a_n$ and for each $a \notin \{a_0, \dots, a_n\}$, we have $s^\perp(a) = 0 = t^\perp(a)$. From $a_0 \le s \le a_n$ we get $(s \wedge a_n) - (s \wedge a_0) = s - a_0$. Thus, by (1),
\[
s - a_0 = \sum_{i=1}^n ((s \wedge a_i) - (s \wedge a_{i-1})) = \sum_{i=1}^n (a_i - a_{i-1})s^\flat(a_i).
\]
So $s = a_0 + \sum_{i=1}^n (a_i - a_{i-1})s^\flat(a_i)$, and a similar argument gives $t = a_0 + \sum_{i=1}^n (a_i - a_{i-1})t^\flat(a_i)$. Finally, if $s, t \ge 0$, then we may choose $a_0$ above to be 0.
\end{proof}

As we showed in Lemma~\ref{lem: * arithmetic}(2), if $s$ and $t$ are elements of a Specker $D$-algebra, then $s \leq t$ iff $s^\flat(a) \leq t^\flat(a)$ for all $a \in D$. 
We strengthen this in the next proposition and show that the analogous property holds for the proximity relation on a proximity 
 Specker $D$-algebra. The desire to have such a simple  functional interpretation of the proximity relation motivates our use of the $(-)^\flat$ representation  of a Specker $D$-algebra in place of the $(-)^\perp$ representation. 

\begin{proposition} \label{prop:4.9}
Let $(S,\lhd)$ be a proximity Specker $D$-algebra and let $s,t \in S$. Then $s \lhd t$ iff $s^\flat(b) \lhd t^\flat(b)$ for all $b \in D$.
\end{proposition}

\begin{proof}
Let $s,t \in S$. We first show that $s\lhd t$ iff  
$[(s-a) \wedge b] \vee 0 \lhd [(t-a)\wedge b] \vee 0$ for all $a,b \in D$. 
First suppose that $s \lhd t$. By Remark~\ref{rem: prox def}(3), $a \lhd a$. Therefore, $s-a \lhd t-a$. Applying Remark~\ref{rem: prox def}(3) again, we first get $(s-a) \wedge b  \lhd (t-a) \wedge b$, and then that $[(s-a) \wedge b]\vee 0 \lhd [(t-a)\wedge b]\vee 0$. Conversely, as $S$ is bounded, there exist $a,b \in D$ with $a \le s,t \le a+b$. Therefore, $[(s-a) \wedge b]\vee 0 = s-a$ and $[(t-a) \wedge b]\vee 0 = t-a$. Thus, $s-a \lhd t - a$. Since $a \lhd a$, we conclude that $s \lhd t$.

Next, let $s \lhd t$ and $b \in D$. Choose $a < b$ so that if $a < c < b$, then $s^\perp(c) = 0 = t^\perp(c)$. By Lemma \ref{* properties}(1), $[(s-a)\wedge (b-a)]\vee 0 = (b-a)s^\flat(b)$ and $[(t-a)\wedge (b-a)] \vee 0 = (b - a)t^\flat(b)$. Consequently, by the previous paragraph, we have $(b-a)s^\flat(b) \lhd (b-a)t^\flat(b)$. Since $b - a > 0$, it follows from (P7) that $s^\flat(b) \lhd t^\flat(b)$.

Conversely, suppose that $s^\flat(b) \lhd t^\flat(b)$ for all $b \in D$. By Lemma \ref{* properties}(2), we may write $s = a_0 + \sum_{i=1}^n (a_i - a_{i-1})s^\flat(a_i)$ and $t = a_0 + \sum_{i=1}^n (a_i - a_{i-1})t^\flat(a_i)$ for appropriate $a_0 < \cdots < a_n$ in $D$. Since $\lhd$ preserves addition and scalar multiplication by nonnegative scalars, from these representations we conclude that $s \lhd t$.
\end{proof}

Let $S$ be a Specker $D$-algebra and $\prec$ be a de Vries proximity on $\func{Id}(S)$. Proposition~\ref{prop:4.9} suggests a way to lift  $\prec$ to a proximity  $\lhd$  on $S$.
 We will show in Corollary~\ref{cor:4.14} that the relation in the following definition is a proximity on $S$ and that it is the unique proximity extending~$\prec$. 

\begin{definition} \label{proximity_definition}
\begin{enumerate}
\item[]
\item Let $\prec$ be a de Vries proximity on a boolean algebra $B$. Define $\prec^\flat$ on $D[B]^\flat$ by $f \prec^\flat g$ if $f(b) \prec g(b)$ for each $b \in D$.
\item Let $S$ be a Specker $D$-algebra and let $\prec$ be a de Vries proximity on $\func{Id}(S)$. Define $\lhd$ on $S$ by $s \lhd t$ if $s^\flat \prec^\flat t^\flat$. 
\end{enumerate}
\end{definition}

\begin{theorem} \label{de Vries lift}
Let $\prec$ be a de Vries proximity on a boolean algebra $B$. Then $\prec^\flat$ is a proximity on $D[B]^\flat$, and is the unique proximity on $D[B]^\flat$ such that $e \prec f$ iff $e^\flat \prec^\flat f^\flat$ for each $e, f \in B$. 
\end{theorem}

\begin{proof}
The proofs of (P1)--(P4) are straightforward. Among the axioms (P5)--(P8), we only verify (P6) since the other axioms follow along similar lines.

(P6) Suppose that $s, t, r, u \in D[B]^\flat$ with $s \prec^\flat t$ and $r \prec^\flat u$. Let $a \in D$. By Theorem~\ref{* axiomitization}(1),
\begin{align*}
(s+r)(a) &= \bigvee_{b_1 + b_2 \ge a} s(b_1) \wedge r(b_2) \\
(t+u)(a) &= \bigvee_{b_1 + b_2 \ge a} t(b_1) \wedge u(b_2)
\end{align*}
Because $\prec$ preserves finite meets and joins (see Remark~\ref{rem: prox def}(3)), it follows that $s+r \prec^\flat t+u$.

(P9) Let $s\prec^\flat t$. By Lemma \ref{* properties}(2), there are $a_0 < \cdots < a_n$ in $D$ and decreasing $e_i, f_i \in B$ with $s(a) = e_i$ and $t(a) = f_i$ for $a_{i-1} < a \le a_i$. From $s \prec^\flat t$ it follows that $e_i \prec f_i$ for each $i$. Since $\prec$ is a de Vries proximity, there is $g_i \in B$ with $e_i \prec g_i \prec f_i$ for each $i$. As the $e_i$ and $f_i$ decrease and $\prec$ preserves finite meets, without loss of generality we may assume that the $g_i$ decrease. Define $r \in D[B]^\flat$ by $r(a) = g_i$ when $a_{i-1} < a \le a_i$. Then $s(a) \prec r(a) \prec t(a)$ for each $a \in D$. Thus, $s \prec^\flat r \prec^\flat t$.

(P10) Let $0 < s$. By Lemma \ref{* properties}(2), there are $a_0 < \cdots < a_n$ in $D$ and $1=f_0>f_1>\cdots>f_n>0$ in $B$ with $s(a) = f_i$ for $a_{i-1} < a \le a_i$. Since $s > 0$ we may assume that $a_0 = 0$. Since $\prec$ is a de Vries proximity, there is $0 < e \in B$ with $e \prec f_n$. Define $t \in D[B]^\flat$ by $t(a) = 1$ if $a \le 0$, $t(a) = e$ if $0 < a \le a_n$, and $t(a) = 0$ if $a_n < a$. Then $t(a) \prec s(a)$ for each $a\in D$, so $t \prec^\flat s$. Also, by Remark~\ref{rem: idempotents of D[B]flat} and Lemma~\ref{lem: * arithmetic}(2), $0 < t$.

Finally, for $e, f \in B$, it follows from Remark~\ref{rem: idempotents of D[B]flat}\ and Proposition~\ref{prop:4.9} that $e \prec f$ iff $e^\flat \prec^\flat f^\flat$, and that $\prec^\flat$ is the unique proximity on $D[B]^\flat$ satisfying this property.
\end{proof}

\begin{corollary} \label{cor:4.14}
Let $S$ be a Specker $D$-algebra and let $\prec$ be a de Vries proximity on $\func{Id}(S)$. If $\lhd$ is the extension of $\prec$ to $S$ given in Definition~\emph{\ref{proximity_definition}(2)}, then $\lhd$ is a proximity on $S$. Furthermore, $\lhd$ is the unique extension of $\prec$ to a proximity on $S$. Consequently, there is a 1-1 correspondence between proximities on $S$ and de Vries proximities on $\Id(S)$.
\end{corollary}

\begin{proof}
Let $S$ be a Specker $D$-algebra and let $\prec$ be a de Vries proximity on $\func{Id}(S)$. By Proposition~\ref{prop: relation between flat and perp}, $(-)^\flat:S\to D[\Id(S)]^\flat$ is an $\ell$-algebra isomorphism. Moreover, for each $s, t \in S$, we have $s \lhd t$ iff $s^\flat \prec^\flat t^\flat$. Therefore, by Theorem~\ref{de Vries lift}, $\lhd$ is the unique proximity on $S$ extending $\prec$. \end{proof}

We recall (see, e.g., \cite[Def.~7.45]{Lam99}) that a commutative ring $R$ is a \emph{Baer ring} if 
for each ideal $I$ of $R$ the annihilator $\{r \in R : rs=0 \ \forall s \in I\}$ of $I$ is a principal ideal generated by an idempotent. By \cite[Cor.~4.4]{BMMO15a}, a Specker $D$-algebra $S$ is a Baer ring iff $\Id(S)$ is a complete boolean algebra. Thus, if $(B,\prec)$ is a de Vries algebra, then $(D[B]^\flat,\prec^\flat)$ is a proximity Specker $D$-algebra and $D[B]^\flat$ is a Baer ring.

\begin{definition}
Let $S$ be a Specker $D$-algebra. If $S$ is a Baer ring, then we call $S$ a {\em Baer-Specker $D$-algebra}. If in addition $\lhd$ is a proximity on $S$, then we call $(S,\lhd)$ a {\em proximity Baer-Specker $D$-algebra}. 
\end{definition} 

We are ready to give a choice-free definition of de Vries powers of $D$.  

\begin{definition}
Let $D$ be a totally ordered domain and $(B, \prec)$ a de Vries algebra. The {\em de Vries power of $D$ by $(B,\prec)$} is the proximity $D$-algebra $(D[B]^\flat, \prec^\flat)$.
\end{definition}

The next theorem shows that  de Vries powers of $D$ are exactly the proximity Baer-Specker $D$-algebras. It was first proved in \cite[Thm.~4.10, Cor.~5.6]{BMMO15b} using choice. Our proof here is choice-free.

\begin{theorem} \label{thm: de Vries power}
\begin{enumerate}[$(1)$]
\item[]
\item If $(B, \prec)$ is a de Vries algebra, then $(D[B]^\flat, \prec^\flat)$ is a proximity Baer-Specker $D$-algebra. 
\item If $(S, \lhd)$ is a proximity Baer-Specker $D$-algebra and $B = \Id(S)$, then $(-)^\flat : S \to D[B]^\flat$ is an $\ell$-algebra isomorphism such that $s \lhd t$ iff $s^\flat \prec^\flat t^\flat$.
\end{enumerate}
\end{theorem}

\begin{proof}
(1) Let $(B, \prec)$ be a de Vries algebra. By Theorem~\ref{* axiomitization}, $D[B]^\flat$ is a Specker $D$-algebra, and $B \cong \Id(D[B]^\flat)$ by Remark~\ref{rem: idempotents of D[B]flat}. Since $B$ is complete, $D[B]^\flat$ is a Baer-Specker $D$-algebra. Thus, the de Vries power $(D[B]^\flat, \prec^\flat)$ is a proximity Baer-Specker $D$-algebra by Theorem~\ref{de Vries lift}.

(2) Let $(S, \lhd)$ be a proximity Baer-Specker algebra and $B = \Id(S)$. The restriction $\prec$ of $\lhd$ to $B$ is a de Vries proximity by Proposition~\ref{prop:4.6}. Since $S$ is Baer, $B$ is a complete boolean algebra. Therefore, $(B, \prec)$ is a de Vries algebra, and so $(D[B]^\flat, \prec^\flat)$ is a de Vries power and $(-)^\flat : S \to D[B]^\flat$ is an $\ell$-algebra isomorphism by Proposition~\ref{prop: relation between flat and perp}. Moreover, $s \lhd t$ iff $s^\flat \prec^\flat t^\flat$ by Proposition~\ref{prop:4.9}. 
\end{proof}

\section{De Vries algebras and proximity Baer-Specker algebras} \label{sec:7}

In this final section we extend the correspondence of Section~\ref{sec:lifting} between de Vries proximities on boolean algebras and proximities on Specker $D$-algebras to a categorical equivalence between the category $\dev$ of de Vries algebras and the category $\PBSp_D$ of proximity Baer Specker $D$-algebras. As we pointed out in the introduction, this equivalence follows from de Vries duality between $\dev$ and $\KHaus$ and the duality of \cite{BMMO15b} between $\PBSp_D$ and $\KHaus$. However, the proof requires going through $\KHaus$ and hence is not choice-free. We give a purely algebraic  choice-free proof of this equivalence.

\begin{definition} \label{def:5.1} \cite[Def.~6.4]{BMMO15b}
Let $(S,\lhd)$ and $(T,\lhd)$ be proximity Baer-Specker $D$-algebras. A {\em proximity morphism} is a map $\alpha : S \to T$ satisfying
\begin{enumerate}
\item[(M1)] $\alpha(0) = 0$.
\item[(M2)] $\alpha(s \wedge t) = \alpha(s) \wedge \alpha(t)$.
\item[(M3)] $s \lhd t$ implies $-\alpha(-s) \lhd \alpha(t)$.
\item[(M4)] $\alpha(t)=\bigvee\{\alpha(s) : s \lhd t\}$.
\item[(M5)] $s \in S$ and $a \in D$ imply $\alpha(s + a) = \alpha(s) + a$.
\item[(M6)] $s \in S$ and $0\le a \in D$ imply $\alpha(as) = a\alpha(s)$.
\item[(M7)] $s \in S$ and $a \in D$ imply $\alpha(s \vee a) = \alpha(s) \vee a$. 
\end{enumerate}
\end{definition}

\begin{remark}
\begin{enumerate}
\item[]
\item It is immediate from (M1) and (M5) that $\alpha(a) = a$ for each $a \in D$.
\item The reading of axiom (M4) should be that the least upper bound of $\{\alpha(s) : s \lhd t\}$ exists and is equal to $\alpha(t)$.  
\item The axioms (M1)--(M4) are direct analogues of the corresponding axioms for de Vries morphisms, while the axioms (M5)--(M7) govern the behavior of proximity morphisms with respect to addition, multiplication, and join by a scalar. 
\end{enumerate}
\end{remark}

It was proved in \cite[Prop.~6.6]{BMMO15b} in a choice-dependent way that a proximity morphism between proximity Baer-Specker $D$-algebras restricts to a de Vries morphism between the de Vries algebras of idempotents. To give a choice-free proof, we require the following lemma, which gives a strictly order-theoretic characterization of idempotents in an $f$-ring. 

\begin{lemma} \label{lem: idempotent characterization}
Let $A$ be an $f$-ring and $e \in A$.  Then $e \in \Id(A)$ iff $e = 2e \wedge 1$.
\end{lemma}

\begin{proof}
Let $e \in A$. We have 
\[
e = 2e \wedge 1 \mbox{ iff } (2e \wedge 1) - e = 0 \mbox{ iff } e \wedge (1-e) = 0.
\] 
Therefore, if $e = 2e \wedge 1$, then $0 \le e, 1-e \le 1$, and hence $e(1-e) = 0$ because $e(1-e) \le e \wedge (1-e)$. Thus, $e^2 = e$. Conversely, let $e \in \Id(A)$. Since $A$ is an $f$-ring, the proof of \cite[Lem.~4.9(2)]{BMMO15b} shows that the order on $\Id(A)$ is the restriction of the order on $A$. Therefore, 
\[
(1\wedge 2e)-e = (1-e)\wedge e = 0. 
\]
Thus, $e = 2e \wedge 1$.
\end{proof}

To prove that each de Vries morphism lifts to a proximity morphism, we need the following lemma. A choice-dependent proof of Item (1) was given in \cite[Prop.~6.6]{BMMO15b}. We give a choice-free proof of (1) which together with \cite[Thm.~6.7]{BMMO15b} then yields a choice-free proof of Item~(2). 

\begin{lemma} \label{prop:5.4}\label{proximity_morphism_on_decr_rep}
Let $(S,\lhd)$ and $(T,\lhd)$ be proximity Baer-Specker $D$-algebras and let $\alpha : S \to T$ be a proximity morphism. 
\begin{enumerate}[$(1)$]
\item $\alpha(\func{Id}(S)) \subseteq \func{Id}(T)$ and $\alpha|_{\func{Id}(S)}$ is a de Vries morphism from $\func{Id}(S)$ to $\func{Id}(T)$.
\item Let $s \in S$ and write $s = a_0 + \sum_{i=1}^n b_i e_i$ in decreasing form. Then $\alpha(s) = a_0 + \sum_{i=1}^n b_i \alpha(e_i)$.
\end{enumerate}
\end{lemma}

\begin{proof}
(1) Let $e \in \Id(S)$. By Lemma~\ref{lem: idempotent characterization}, $e = 2e \wedge 1$, so
\[
\alpha(e) = \alpha(2e \wedge 1) = \alpha(2e) \wedge \alpha(1) = 2\alpha(e) \wedge 1.
\]
Therefore, $\alpha(e) \in \Id(T)$.

It follows that $\alpha|_{\func{Id}(S)}: \func{Id}(S) \to \func{Id}(T)$ is well defined. It is also clear that $\alpha|_{\func{Id}(S)}$ satisfies (M1) and (M2). Suppose that $e,f\in\func{Id}(S)$ with $e \lhd f$. Since $\lnot e=1-e$, we have
\[
\lnot \alpha(\lnot e)=1-\alpha(1-e)=1-[1+\alpha(-e)]=-\alpha(-e). 
\]
Because $-\alpha(-e) \lhd \alpha(f)$, we conclude that $\lnot \alpha(\lnot e) \lhd \alpha(f)$. Therefore, $e \lhd f$ implies that $\lnot \alpha(\lnot e) \lhd \alpha(f)$, which is the de Vries analogue of (M3). Let $f \in \func{Id}(S)$. Then $\alpha(f) = \bigvee \{\alpha(s) : s \in S, s \lhd f\}$. Suppose that $0 \le s \lhd f$. As above, write $s = \sum_{i=1}^n b_i e_i$ in orthogonal form with each $b_i>0$. Then $b_ie_i \le s \lhd f$, so $b_ie_i \lhd f$ by (P3). It follows from the proof of \cite[Prop.~5.1]{BMMO15b}, which uses (P7), that $b_i \le 1$ and $e_i \lhd f$ for each $i$. Consequently, $s \le e_1 \vee \cdots \vee e_n \lhd f$. Since $\alpha(s) \le \alpha(e_1 \vee \cdots \vee e_n)$ and $e_1 \vee \cdots \vee e_n \in \Id(S)$, we see that $\alpha(f)=\bigvee\{\alpha(e) : e \in \func{Id}(S), e \lhd f\}$. Thus, $\alpha|_{\func{Id}(S)}$ satisfies (M4).

(2) In the proof of \cite[Thm.~6.7]{BMMO15b} substitute the choice-dependent proof of (1) with the choice-free proof above.
\end{proof}

We recall the isomorphism $\tau_B : B \to \Id(D[B]^\flat)$ from Remark~\ref{rem: idempotents of D[B]flat}, given by $\tau_B(e) = e^\flat$.

\begin{theorem} \label{morphism_uniqueness}
Let $(A,\prec)$ and $(B,\prec)$ be de Vries algebras and let $\sigma : A \to B$ be a de Vries morphism. Then there is a unique proximity morphism $\sigma^\flat : D[A]^\flat \to D[B]^\flat$ such that $\sigma^\flat \circ \tau_A = \tau_B \circ \sigma$.
\[
\begin{tikzcd}
A \arrow[r, "\sigma"] \arrow[d, "\tau_A"'] & B \arrow[d, "\tau_B"] \\
D[A]^\flat \arrow[r, "\sigma^\flat"'] & D[B]^\flat
\end{tikzcd}
\]
\end{theorem}

\begin{proof}
Define $\sigma^\flat: D[A]^\flat \to D[B]^\flat$ by $\sigma^\flat(f) = \sigma\circ f$. It is easy to see that $\sigma^\flat$ is well defined. Let $e \in A$ and consider the corresponding idempotent $\tau_A(e) = e^\flat \in D[A]^\flat$. It follows from Remark~\ref{rem: idempotents of D[B]flat} that $\sigma\circ e^\flat = \sigma(e)^\flat$. Thus, $\sigma^\flat \circ \tau_A = \tau_B \circ \sigma$. 

We now show that $\sigma^\flat$ is a proximity morphism. Verifying (M1) and (M2) is straightforward, so we begin with (M3). 

(M3)
We first show that
\begin{equation}
(-f)(a) = \bigwedge \{\lnot f(b) : b > -a\}. \tag{$*$} \label{eqn: -f}
\end{equation}
for each $f \in D[A]^\flat$ and $a \in D$. 
There is $s \in D[A]^*$ with $f = \alpha(s)$. Since $\alpha$ is a $D$-algebra homomorphism, $-f = \alpha(-s)$. Therefore,
\begin{align*}
(-f)(a) &= \bigvee \{ (-s)(b) : b \ge a\} = \bigvee \{ s(-b) : b \ge a \} = \bigvee \{ s(c) : c \le -a\}.
\end{align*}
Since the values of $s$ are pairwise orthogonal and $\bigvee \{ s(c) : c \in D\} = 1$ , we have
\[
(-f)(a) = \lnot \bigvee \{ s(c) : c > -a\} = \bigwedge \{ \lnot s(c) : c > -a\}.
\]
In addition,
\begin{align*}
\bigwedge  \{\lnot f(b) : b > -a\} &= \bigwedge \{ \lnot \bigvee \{ s(c) : c \ge b\} : b > -a\} \\
&= \bigwedge \{ \lnot s(c) : c \ge b > -a\} = \bigwedge \{ \lnot s(c) : c > -a\}.
\end{align*}
This verifies Equation~(\ref{eqn: -f}).

Now let $f,g \in D[A]^\flat$ with $f \prec^\flat g$ and let $a \in D$. By Equation~(\ref{eqn: -f}) applied twice,
\[
(-\sigma^\flat(-f))(a) = \bigwedge_{b > -a} \lnot \sigma^\flat(-f)(b) =\bigwedge_{b > -a} \lnot \sigma((-f)(b)) = \bigwedge_{b > -a} \lnot \sigma\left(\bigwedge_{c > -b} \lnot f(c)\right). 
\] 
Thus, as $\sigma$ preserves finite meets,
\begin{align*}
(-\sigma^\flat(-f))(a) &= \bigwedge_{b > -a} \lnot \sigma\left(\bigwedge_{c > -b} \lnot f(c)\right)  \\
&= \bigwedge_{b > -a} \lnot \bigwedge_{c > -b} \sigma(\lnot f(c)) = \bigwedge_{b > -a} \bigvee_{c > -b} \lnot \sigma(\lnot f(c)) \\
&= \bigwedge_{a > d} \bigvee_{c > d} \lnot \sigma(\lnot f(c)).
\end{align*}
Since $f \prec^\flat g$, we have $f(c) \prec g(c)$ for each $c \in D$, and so $\lnot \sigma(\lnot f(c)) \prec \sigma(g(c))$. Hence, $\bigwedge_{a > d} \bigvee_{c > d} \lnot \sigma(\lnot f(c)) \prec \bigwedge_{a > d} \bigvee_{c > d} \sigma(g(c))$. We show that $\bigwedge_{a > d} \bigvee_{c > d} \sigma(g(c)) = \sigma(g(a))$. There are $a_0 < \cdots < a_n$ in $D$ and $1 = e_0 > \cdots > e_n > 0$ in $A$ with $g(a) = e_i$ if $a_{i-1} < a \le a_i$, $g(a) = 1$ if $a \le a_0$, and $g(a) = 0$ if $a > a_n$. Let $a \in D$. There is $i$ with $a_{i-1} < a \le a_i$. Take $d < a$. If $a_{i-1} \le d$, then $\bigvee_{c > d}\sigma(g(c)) = \sigma(e_i) = \sigma(g(a))$ since $e_i$ is the largest element of $\{g(c) : c > d\}$. On the other hand, if $d < a_{i-1}$, then $\bigvee_{c > d} \sigma(g(c)) = \sigma(e_{i-1}) \ge \sigma(e_i)$. Thus, $\bigwedge_{a > d} \bigvee_{c > d} \sigma(g(c)) = \sigma(e_i) = \sigma(g(a))$, as desired. Consequently, we have  $(-\sigma^\flat(-f))(a) \prec \sigma^\flat(g)(a)$ for each $a \in D$, which yields $(-\sigma^\flat(-f)) \prec^\flat \sigma^\flat(g)$.

(M4) Let $g \in D[A]^\flat$. Clearly $\sigma^\flat(g)$ is an upper bound of $\{\sigma^\flat(f) : f \prec^\flat g\}$. To see that $\sigma^\flat(g)$ is the least upper bound, it is sufficient to show that
\begin{equation}
\sigma(g(a)) = \bigvee \{ \sigma(f(a)) : f \prec^\flat g\} \textrm{ for each }a \in D. \tag{$**$} \label{eqn: M4}
\end{equation}
Indeed, suppose Equation~(\ref{eqn: M4}) holds and $h \in D[B]^\flat$ is an upper bound of $\{\sigma^\flat(f) : f \prec^\flat g\}$.  Then, by Lemma~\ref{* arithmetic}(2), $h(a) \ge \sigma^\flat(f)(a) = \sigma(f(a))$ for each $a \in D$. Therefore, $h(a) \ge \sigma(g(a)) = \sigma^\flat(g)(a)$ for each $a \in D$, and so $h \ge \sigma^\flat(g)$, again by Lemma~\ref{* arithmetic}(2).

To prove Equation~(\ref{eqn: M4}), there are $a_0 < \cdots < a_n$ in $D$ and $1 = f_0 > \cdots > f_n > 0$ in $A$ with $g(a) = f_i$ if $a_{i-1} < a \le a_i$,  $g(a) = 1$ if $a \le a_0$, and $g(a) = 0$ if $a > a_n$. If $1 = e_0 \ge e_1 \ge \cdots \ge e_n$ are elements of $A$, then there is $f \in D[A]^\flat$ satisfying $f(a) = 1$ if $a \le a_0$, $f(a) = e_i$ if $a_{i-1} < a \le a_i$ for each $i$, and $f(a) = 0$ if $a > a_n$. We call $f$ the function associated to $e_0 \ge \cdots \ge e_n$ (fixing $a_0 < \cdots < a_n)$. Note that $f \prec^\flat g$ iff $e_i \prec f_i$ for each $i$ by the definition of $\prec^\flat$.

Let $a \in D$. If $a \le a_0$, then $\sigma(g(a)) = \sigma(1) = 1$. Let $f \in D[A]^\flat$ be associated to $1 > 0 \ge 0 \ge\cdots \ge 0$. Then $f \prec^\flat g$ and $f(a) = 1$, so $\sigma(f(a)) = 1$. Therefore, Equation~(\ref{eqn: M4}) holds for $a \le a_0$. Next, suppose that $a > a_n$. Then $\sigma(g(a)) = \sigma(0) = 0$. With the same $f$, we have $\sigma(f(a)) = 0$, and so Equation~(\ref{eqn: M4}) also holds for $a > a_n$. Finally, let $a_{i-1} < a \le a_i$. Since $\sigma$ is a de Vries morphism, $\sigma(g(a)) = \bigvee \{ e : e \prec g(a)\}$. Let $e \in A$ with $e \prec g(a)$ and let $f$ be the function associated to $1 \ge  \cdots \ge 1 \ge e \ge 0 \ge \cdots \ge 0$, where $e$ is the $i$-th term in the sequence. Then $f \prec^\flat g$ and $f(a) = e$. Since we can produce $f \prec^\flat g$ for each $e$ with $e \prec g(a)$, this shows that Equation~(\ref{eqn: M4}) holds for $a$. Therefore, Equation~(\ref{eqn: M4}) holds for all $a \in D$.  

The proofs of (M5)--(M7) are similar to each other, so we only give the proof of (M6).

(M6) Let $f \in D[A]^\flat$ and $0 < b \in D$. Let $\{1 = e_0 > \cdots > e_n > 0\}$ be the image of $f$ in $A$. For each $i$, let $a_i$ be the largest element of $f^{-1}(e_i)$. We claim that $(bf)(a) = e_i$ for $ba_{i-1} < a \le ba_i$. For suppose that $ba_{i-1} < a \le ba_i$. By Theorem~\ref{* axiomitization}(2), $(bf)(a) = \bigvee_{bc \ge a} f(c)$. Since $bc \ge a > ba_{i-1}$, 
we have $b(c-a_{i-1}) = bc - ba_{i-1} > 0$. If $c \le a_{i-1}$, then $c - a_{i-1} \le 0$, so $b(c - a_{i-1}) \le 0$, a contradiction. Therefore,
$c > a_{i-1}$, and hence $f(c) \le e_i$. Thus, by choosing $c = a_i$, we see that the join is $f(a_i) = e_i$, so that  $(bf)(a) = e_i$ for $ba_{i-1} < a \le ba_i$, as claimed. We see then that $\sigma^\flat(bf)(a) = \sigma((bf)(a)) = \sigma(e_i)$ if $ba_{i-1} < a \le ba_i$.
By the same reasoning, since $\{1 = \sigma(e_0) > \cdots > \sigma(e_n) > 0\}$ is the image of $\sigma^\flat(f)$, we have $(b\sigma^\flat(f))(a) = \sigma(e_i)$ if $ba_{i-1} < a \le ba_i$.
Therefore, $\sigma^\flat(bf) = b\sigma^\flat(f)$.

Finally, to prove uniqueness, suppose $\gamma : D[A]^\flat\to D[B]^\flat$ is another proximity morphism with $\gamma \circ \tau_A = \tau_B \circ \sigma$. Let $f \in D[A]^\flat$. Then there are $a_0 < \cdots < a_n$ in $D$ and $1 = e_0 > \cdots > e_n > 0$ in $A$ with $f(a) = e_i$ when $a_{i-1} < a \le a_i$. By Proposition~\ref{decr dec}(2), if $b_i = a_i - a_{i-1}$ for $1 \le i \le n$, then $f = a_0 + \sum_i b_i e^\flat_i$. By Lemma~\ref{proximity_morphism_on_decr_rep}(2), $\gamma(f) = a_0 + \sum_i b_i \sigma(e_i)^\flat$. Therefore, $\gamma(f)(a) = \sigma(e_i)$ when $a_{i-1} < a \le a_i$. Thus, $\gamma(f) = \sigma \circ f  = \sigma^\flat(f)$.
\end{proof}

\begin{definition} \label{def: eta}
For a Specker $D$-algebra $S$ let  $\eta_S : S \to D[\Id(S)]^\flat$ be given by $\eta_S(s) = s^\flat$. 
\end{definition}

By Theorem~\ref{prop: relation between flat and perp}, $\eta_S$ is an  $\ell$-algebra isomorphism. 

\begin{corollary} \label{cor:7.8}
If $(S,\lhd)$ and $(T,\lhd)$ are proximity Baer-Specker $D$-algebras and $\sigma : \Id(S) \to \Id(T)$ is a de Vries morphism, then there is a unique proximity morphism $\alpha : S \to T$ such that $\alpha$ extends $\sigma$ and $\sigma^\flat \circ \eta_S = \eta_T \circ \sigma$.
\[
\begin{tikzcd}
S \arrow[d, "\eta_S"'] \arrow[r, "\alpha"] & T \arrow[d, "\eta_T"] \\
D[\Id(S)]^\flat \arrow[r, "\sigma^\flat"'] & D[\Id(T)]^\flat
\end{tikzcd}
\]

\end{corollary}

\begin{proof}
By Theorem~\ref{morphism_uniqueness}, there is a unique proximity morphism $\sigma^\flat : D[\Id(S)]^\flat \to D[\Id(T)]^\flat$ such that $\sigma^\flat \circ \tau_{\Id(S)} = \tau_{\Id(T)} \circ \sigma$. 
\[
\begin{tikzcd}
\Id(S) \arrow[r, "\sigma"] \arrow[d, hookrightarrow] \arrow[dd, bend right = 40, "\tau_{\Id(S)}"'] & \Id(T) \arrow[d, hookrightarrow]  \arrow[dd, bend left = 40, "\tau_{\Id(T)}"] \\
S \arrow[d, "\eta_S"] \arrow[r, "\alpha"] & T \arrow[d, "\eta_T"'] \\
D[\Id(S)]^\flat \arrow[r, "\sigma^\flat"'] & D[\Id(T)]^\flat
\end{tikzcd}
\]
Set $\alpha = \eta_T^{-1}\circ \sigma^\flat \circ \eta_S$. Then $\alpha$ is a proximity morphism and $\sigma^\flat \circ \eta_S = \eta_T \circ \sigma$. Let $e \in \Id(S)$. We have $\alpha(e) = \eta_T^{-1}\sigma^\flat\eta_S(e) = \eta_T^{-1}\eta_T\sigma(e) = \sigma(e)$, hence $\alpha$ extends $\sigma$. Finally, since $\sigma^\flat$ is unique, so is $\alpha$.
\end{proof}

Let $\sigma_1 : (B_1, \prec) \to (B_2, \prec)$ and $\sigma_2 : (B_2, \prec) \to (B_3, \prec)$ be de Vries morphisms. We recall that the composition $\sigma_2 \star \sigma_1$ in $\dev$ is defined by 
\[
(\sigma_2 \star \sigma_1)(e) = \bigvee \{ \sigma_2\sigma_1(f) : f \prec e \}.
\]

\begin{theorem}
Proximity Baer-Specker $D$-algebras and proximity morphisms between them form a category $\PBSp_D$, where the composition $\alpha_2\star\alpha_1$ of two proximity morphisms $\alpha_1:S_1\to S_2$ and $\alpha_2:S_2\to S_3$ is the unique proximity morphism extending the de Vries morphism $\alpha_2|_{\Id(S_2)} \star \alpha_1|_{\Id(S_1)}$. It is given by
\[
(\alpha_2\star\alpha_1)(s)=\bigvee\{\alpha_2\alpha_1(t) : t\lhd s\}.
\]
\end{theorem}

\begin{proof}
Let $\alpha_1:S_1\to S_2$ and $\alpha_2:S_2\to S_3$ be proximity morphisms. By Proposition~\ref{prop:5.4}, their restrictions to the idempotents are de Vries morphisms. Therefore, $\alpha_2|_{\Id(S_2)}\star\alpha_1|_{\Id(S_1)}$ is a de Vries morphism. By Corollary~\ref{cor:7.8}, $\alpha_2\star\alpha_1$ is the unique proximity morphism extending $\alpha_2|_{\Id(S_2)}\star\alpha_1|_{\Id(S_1)}$. That $\star$ is associative follows from Corollary~\ref{cor:7.8} and the fact that de Vries composition is associative. Since identity morphisms are identity functions, it is then clear that $\PBSp_D$ forms a category.

It is left to show that $(\alpha_2\star\alpha_1)(s)=\bigvee\{\alpha_2\alpha_1(t) : t\lhd s\}$.
First, suppose that $t \lhd s$. By Lemma~\ref{* properties}(2), write $s = a_0 + \sum_{i=1}^n b_i s^\flat(a_i)$ and $t = a_0 + \sum_{i=1}^n b_i t^\flat(b_i)$. Set $e_i = s^\flat(a_i)$ and $f_i = t^\flat(a_i)$. Then $f_i \lhd e_i$ for each $i$ by Theorem~\ref{prop:4.9}. We have $\alpha_2\alpha_1(t) =  a_0 + \sum_{i=1}^n b_i \alpha_2\alpha_1(f_i)$. Also, since $\alpha_2\star\alpha_1$ is a proximity morphism, $(\alpha_2 \star \alpha_1)(s) = a_0 + \sum_i b_i (\alpha_2\star\alpha_1)(e_i)$ by Lemma~\ref{proximity_morphism_on_decr_rep}(2). As $\alpha_2|_{\Id(S_2)}\star\alpha_1|_{\Id(S_1)}$ is a de Vries morphism, $(\alpha_2\star\alpha_1)(e_i) = \bigvee \{ \alpha_2\alpha_1(e) : e \in \Id(S_1), e \lhd e_i \}$, so $\alpha_2\alpha_1(t) \le (\alpha_2 \star \alpha_1)(s)$. Therefore, $(\alpha_2\star\alpha_1)(s)$ is an upper bound of $\{\alpha_2\alpha_1(t) : t\lhd s\}$. To see that it is the least upper bound, let $r$ be an upper bound of $\{ \alpha_2\alpha_1(t) : t \lhd s\}$. Let $E_i = \{ e : e \lhd e_i \}$. By \cite[Eqn.~XIII.3(8)]{Bir79} and Lemma~\ref{proximity_morphism_on_decr_rep}(2),
\begin{align*}
(\alpha_2\star\alpha_1)(s) &= a_0 + \sum_{i=1}^n b_i (\alpha_2\star\alpha_1)(e_i) = a_0 + \sum_{i=1}^n b_i \bigvee \{\alpha_2\alpha_1(e) : e \in E_i\} \\
&= \bigvee \{ a_0 + \sum_{i=1}^n b_i \alpha_2\alpha_1(k_i) : k_i \in E_i, 1 \le i \le n \} = \bigvee \{ \alpha_2\alpha_1(a_0 + \sum_{i=1}^n b_i k_i) : k_i \in E_i \}.
\end{align*}
Since $k_i \lhd e_i$ for each $i$, we have $a_0 + \sum_{i=1}^n b_i k_i \lhd s$ by (P6) and (P7). Therefore, $(\alpha_2\star\alpha_1)(s) = \bigvee \{ \alpha_2\alpha_1(a_0 + \sum_{i=1}^n b_i k_i) : k_i \in E_i \} \le r$.
Thus, $(\alpha_2\star\alpha_1)(s)=\bigvee\{\alpha_2\alpha_1(t) : t\lhd s\}$.
\end{proof}

Although proximity morphisms are not in general $D$-algebra homomorphisms, as was shown in \cite[Lem.~8.3]{BMMO15b} with a choice-free proof, proximity isomorphisms are $D$-algebra isomorphisms that preserve and reflect proximity. This is similar to what happens in $\dev$ \cite[Prop.~I.5.5]{deV62}.

\begin{lemma}\label{proximity isomorphism} \cite[Lem.~8.3]{BMMO15b}
Let $(S,\lhd),(T,\lhd)\in\PBSp_D$ and let $\alpha : S \to T$ be a proximity morphism. Then $\alpha$ is an isomorphism in $\PBSp_D$ iff $\alpha$ is a $D$-algebra isomorphism such that $s \lhd t$ in $(S,\lhd)$ iff $\alpha(s) \lhd \alpha(t)$ in $(T,\lhd)$.
\end{lemma}

We are finally ready to give a choice-free proof that $\PBSp_D$ is equivalent to $\dev$.

\begin{theorem} [Main Theorem] \label{equiv} 
The category $\PBSp_D$ of proximity Baer-Specker $D$-algebras is equivalent to the category $\dev$ of de Vries algebras.
\end{theorem}

\begin{proof}
Define a functor $\Id:\PBSp_D\to\dev$ by sending $(S,\lhd)\in\PBSp_D$ to the de Vries algebra $(\Id(S),\lhd|_{\Id(S)})$ and a proximity morphism $\alpha : S \to T$ to the de Vries morphism $\alpha|_{\Id(S)}$. 
It follows from Proposition~\ref{prop:4.6}, Lemma~\ref{prop:5.4}, and the definition of compositions in $\PBSp_D$ and $\dev$ that $\Id$ is well defined.

Define a functor $\Sp:\dev \to \PBSp_D$ by sending $(B,\prec)\in\dev$ to the de Vries power $D$-algebra $(D[B]^\flat,\prec^\flat)$ and a de Vries morphism $\sigma : A \to B$ to the proximity morphism $\sigma^\flat$. 
It follows from Theorem~\ref{thm: de Vries power}, Theorem~\ref{morphism_uniqueness}, and the definition of compositions in $\PBSp_D$ and $\dev$ that $\Sp$ is well defined.

To see that $\Id$ and $\Sp$ form an equivalence, we show that there are natural isomorphisms $\eta : 1_{\PBSp_D} \to \Sp \circ \Id$ and $\tau : \Sp \circ \Id \to 1_{\PBSp_D}$. We define $\eta$ by letting its components be $\eta_S : S \to D[\Id(S)]^\flat$, where $\eta_S(s) = s^\flat$ (see Definition~\ref{def: eta}).  To see that $\eta$ is a natural isomorphism, let $\alpha : (S_1, \lhd) \to (S_2, \lhd)$ be a proximity morphism. Set $B_i = \Id(S_i)$ and $\sigma = \alpha|_{B_1}$. Then $\sigma^\flat = \Sp(\Id(\alpha))$, and we have the following diagram,
\[
\begin{tikzcd}[column sep = 5pc]
S_1 \arrow[r, "\alpha"] \arrow[d, "\eta_{S_1}"'] & S_2 \arrow[d, "\eta_{S_2}"] \\
D[B_1]^\flat \arrow[r, "\Sp(\Id(\alpha))"'] & D[B_2]^\flat
\end{tikzcd}
\]
which commutes by Corollary~\ref{cor:7.8}. Thus, $\eta$ is a natural transformation, and it is then a natural isomorphism by Proposition~\ref{prop: relation between flat and perp}.

We define $\tau : 1_{\dev} \to \Id \circ \Sp$ by letting its components be $\tau_B : B \to \Id(D[B]^\flat)$, where $\tau_B(e) = e^\flat$ (see Remark~\ref{rem: idempotents of D[B]flat}). To see that $\tau$ is a natural isomorphism, let $\sigma : (B_1, \prec) \to (B_2, \prec)$ be a de Vries morphism.  We have the following diagram,
\[
\begin{tikzcd}[column sep = 5pc]
B_1 \arrow[d, "\tau_{B_1}"'] \arrow[r, "\sigma"] & B_2 \arrow[d, "\tau_{B_2}"] \\
\Id(D[B_1]^\flat) \arrow[r, "\Id(\Sp(\sigma))"'] & \Id(D[B_2]^\flat)
\end{tikzcd}
\]
which commutes by Theorem~\ref{morphism_uniqueness}. Thus, $\tau$ is a natural transformation, and it is then a natural isomorphism by Remark~\ref{rem: idempotents of D[B]flat}.
Consequently, $\func{Sp}$ and $\Id$ establish an equivalence of $\PBSp_D$ and $\dev$.
\end{proof}

\def\cprime{$'$}
\begin{remark}
Let $A$ be an algebra of a fixed type. Generalizing \cite[p.~5]{Fuc63}, we say that a binary relation $R$ on $A$ is \emph{compatible} with the operations of $A$ if for each $n$-ary operation $\lambda$ on $A$ there is a subalgebra $B$ of $A$ such that from $a_1 \mathrel{R} b_1, \dots,  a_n \mathrel{R} b_n$ it follows that $\lambda(a_1,\dots,a_n) \mathrel{R} \lambda(b_1,\dots, b_n)$ or $\lambda(b_1,\dots, b_n) \mathrel{R} \lambda(a_1,\dots,a_n)$ for each $a_1,\dots, a_n, b_1, \dots, b_n \in B$. Such a pair $(A,R)$ is a particular case of an algebraic system of Malcev \cite{Mal73}. Let $B$ be a boolean algebra and let $r$ be a binary relation on $B$. We let $A[B]^*$ be the 
boolean power of $A$, as defined by Foster and discussed in Section~\ref{sec:specker}. Define a relation ${\mathcal R}$ on $A[B]^*$ by $$f \: {\mathcal R} \: g \mbox{ iff } \displaystyle{\bigvee} \{f(b):a \mathrel{R} b\} \mathrel{r} \displaystyle{\bigvee} \{g(b):a \mathrel{R} b\} \mbox{ for all } a \in A.$$ Then ${\mathcal{R}}$ lifts $r$ and $R$ to the boolean power $A[B]^*$. If $A$ is a totally ordered domain, $R$ is $\le$ and $r$ is $\prec$, then this generalization of a 
boolean power is exactly our de Vries power. It would be interesting  to study in more detail this generalization of 
boolean powers when additional relations are also at play. Of course, the binary relations $R$ and $r$ can further be generalized to arbitrary relations. A particular case of such a generalization, when $R$ is present but $r$ is not, is briefly discussed by Banaschewski and Nelson \cite[Concluding Remarks]{BN80}. 
\end{remark}

%

\def\cprime{$'$}
\providecommand{\bysame}{\leavevmode\hbox to3em{\hrulefill}\thinspace}
\providecommand{\MR}{\relax\ifhmode\unskip\space\fi MR }
\providecommand{\MRhref}[2]{%
  \href{http://www.ams.org/mathscinet-getitem?mr=#1}{#2}
}
\providecommand{\href}[2]{#2}

\end{document}